\documentclass[11pt]{amsart}
\usepackage[top=30truemm,bottom=30truemm,left=30truemm,right=30truemm]{geometry} 
\usepackage{amsmath, amssymb, array}
\usepackage{mathtools}
\usepackage[all]{xy}
\usepackage{ascmac}
\usepackage[dvips]{graphicx}
\usepackage{color}
\usepackage{amscd}
\usepackage{fancyhdr}
\usepackage[bookmarks=false]{hyperref}
\usepackage{amsrefs}
\usepackage{cleveref}
\usepackage{chngcntr}
\usepackage{apptools}
\usepackage[titletoc,title]{appendix}
\newcommand{\ol}{\overline}

\newcommand{\rank}{\mathrm{rank}}
\newcommand{\Sel}{\mathrm{Sel}^{(2)}}
\newcommand{\Cl}{\mathrm{Cl}}
\newcommand{\Ker}{\mathrm{Ker}}
\newcommand{\res}{\mathrm{res}}
\newcommand{\val}{\mathrm{val}}
\newcommand{\disc}{\mathrm{disc}}
\newcommand{\tr}{\mathrm{tr}}
\renewcommand{\1}{\mathbf{1}}
\newcommand{\fn}{\footnote}

\title{Infinitely many hyperelliptic curves with exactly two rational points} 

\author{Yoshinosuke Hirakawa \and Hideki Matsumura}
\address[Yoshinosuke Hirakawa]{Department of Science and Technology, Keio University, 14-1, Hiyoshi 3-chome, Kouhoku-ku, Yokohama-shi, Kanagawa-ken, Japan}
\email{hirakawa@keio.jp}
\address[Hideki Matsumura]{Department of Science and Technology, Keio University, 14-1, Hiyoshi 3-chome, Kouhoku-ku, Yokohama-shi, Kanagawa-ken, Japan}
\email{hidekimatsumura@keio.jp}

\thanks{This research was supported in part by KAKENHI 26247004, 18H05233 as well as the JSPS Core-to-Core program ``Foundation of a Global Research Cooperative Center in Mathematics focused on Number Theory and Geometry" and the KiPAS program FY2014--2018 of the Faculty of Science and Technology at Keio University. Both authors were supported by the Research Grant of Keio Leading-edge Laboratory of Science $\&$ Technology 2018--2019
 (Grant Numbers 000036 and 000053).}
\date{\today}

\makeatletter
  \def\section{\@startsection{section}{1}{\z@}%
     {4ex plus 0ex}%
     {4ex plus 0ex}%
     {\normalfont\large\bfseries}}
  \makeatother

\makeatletter
  \def\subsection{\@startsection{subsection}{1}{\z@}%
     {2ex plus 0ex}%
     {2ex plus 0ex}%
     {\normalfont\normalsize\bfseries}}
  \makeatother

\theoremstyle{plain}
 \newtheorem{theorem}{Theorem}[section] 
  \crefname{theorem}{Theorem}{Theorems}
 \newtheorem{proposition}[theorem]{Proposition}
 \crefname{proposition}{Proposition}{Propositions}
 \newtheorem{lemma}[theorem]{Lemma}
 \crefname{lemma}{Lemma}{Lemmas}
 \newtheorem{corollary}[theorem]{Corollary}
  \crefname{corollary}{Corollary}{Corollaries}
 
   \crefname{conjecture}{Conjecture}{Conjectures}
 
 \crefname{question}{Question}{Questions}
 
   \crefname{problem}{Problem}{Problems}
 \newtheorem{notation}[theorem]{Notation}
    \crefname{notation}{Notation}{Notations}
 
\theoremstyle{definition} 
 
  \crefname{definition}{Definition}{Definitions}
 
   \crefname{example}{Example}{Examples}
 \newtheorem{remark}[theorem]{Remark}
   \crefname{remark}{Remark}{Remarks}

\begin{document}


\maketitle


\begin{abstract}
In this paper, we construct some families of infinitely many hyperelliptic curves of genus $2$ with exactly two rational points. In the proof, we first show that the Mordell-Weil ranks of these hyperelliptic curves are $0$
and then determine the sets of rational points by using the Lutz-Nagell type theorem for hyperelliptic curves which was proven by Grant.   
\end{abstract}

\tableofcontents

\section{Main theorem}  

In this paper, we determine the set of rational points of the following hyperelliptic curves.

\begin{theorem} \label{MT}
Let $p$ be a prime number, $i, j \in {\mathbb{Z}}$, and $C^{(p;i,j)}$ be a hyperelliptic curve defined by
\[y^2=x(x^2+2^ip^{j})(x^2+2^{i+1}p^{j}).\] 
Suppose that one of the following conditions holds.
\begin{enumerate}
\item $p \equiv 3 \pmod{16}$ and $(i,j)=(0,1)$.
\item $p \equiv 11 \pmod{16}$ and $(i,j)=(1,1)$.
\item $p \equiv  3 \pmod{8}$ and $(i,j)=(0,2)$.
\item $p \equiv - 3 \pmod{8}$ and $(i,j)=(0,2)$.
\end{enumerate}
Then, $C^{(p;i,j)}(\mathbb{Q})=\{(0,0), \infty\}$. 
\end{theorem}
Here, note that, in each case, there exist infinitely many prime numbers satisfying the above congruent condition by Dirichlet's theorem on arithmetic progressions.
A striking feature of \cref{MT} is that we can treat infinitely many prime numbers $p$ which are ``$2$-adically near to each other" simultaneously.

Our proof of \cref{MT} is based on natural generalizations of the $2$-descent argument and the Lutz-Nagell theorem (cf. \cite{Stoll2001}, \cite[Theorem 3]{Grant}).
Recall that, in the case of elliptic curves, the $2$-descent argument makes it possible to bound the Mordell-Weil rank of an elliptic curve by means of the $2$-Selmer group, and
the Lutz-Nagell theorem makes it possible to determine the torsion points of an elliptic curve by means of the discriminant.
For example, by applying them, we can prove that the only rational points on an elliptic curve defined by $y^{2} = x(x+p)(x-p)$ with a prime number 
$p \equiv 3 \mod{8}$ are $(x, y) = (0, 0), (p, 0), (-p, 0)$ and $\infty$, i.e., such a prime number $p$ is never a congruent number (cf. \cite[D27]{Guy} and the references therein). 
Note that the $2$-descent argument allows us to treat prime numbers $p$ which are ``$2$-adically near to each other" simultaneously.

In  \S $2$, we prove that the Mordell-Weil rank of the Jacobian variety $J^{(p; i, j)}$ of $C^{(p; i, j)}$ is $0$ by the $2$-descent argument \cite{Stoll2001}. 
\fn{The primary advantage of our families is that the verification of this fact is relatively easy.
}
Therefore, it is sufficient to determine the set of rational points on $C^{(p; i, j)}$ which map to torsion points of $J^{(p; i, j)}$ via the Abel-Jacobi map associated with the point at infinity $\infty$.

In \S $3$, we carry out this task by applying the Lutz-Nagell type theorem for hyperelliptic curves which was proven by Grant \cite{Grant}. 
\fn{Since $\# C^{(p; i, j)}(\mathbb{F}_q) \geq 4$ for every $p$, $i$, $j$ and another prime number $q \neq 2$, $p$, 
the Chabauty-Coleman method (cf. \cite[Corollary 4a]{Coleman Duke}, \cite[Theorem 5.3 (b)]{MP}) is not sufficient.} 

\begin{remark}
\begin{enumerate}
\item If $p$ is a prime number which does not satisfy the above congruence condition, then $C^{(p;i,j)}$ may have additional rational points.
 For instance, $C^{(17;0,1)}$ has additional rational points $(8, \pm 252)$.
\item 
Let $i, j\in {\mathbb{Z}}_{\geq 0}$ and $C^{(i, j)}$ be a hyperelliptic curve over $\mathbb{Q}(t)$ defined by
\[y^2=x(x^2+2^it^{j})(x^2+2^{i+1}t^{j}). \]
Then, $C^{(i, j)}(\mathbb{Q}(t))=\{(0,0),\infty\}$. 
For $(i,j)=(0,1)$, $(1,1)$, $(0,2)$, this is a consequence of \cref{MT}.
In fact, we can deduce it easily from the abc conjecture for polynomials (cf. \cite[Theorem]{Snyder}) for general $(i,j)$ (Appendix A).
\end{enumerate}
\end{remark}


\section{$2$-descent}

Let $p$ be a (odd) prime number, $i, j \in {\mathbb{Z}}$, and $f(x)=x(x^2+2^ip^{j})(x^2+2^{i+1}p^{j})$.
Let $ C^{(p; i, j)}$ be a hyperelliptic curve defined by $y^2=f(x)$ and $J^{(p; i, j)}$ be its Jacobian variety.
In this section, we prove the following theorem.
\begin{theorem}  \label{rank=0}
Suppose that one of the following conditions holds.
\begin{enumerate}
\item $p \equiv 3 \pmod{16}$ and $(i,j)=(0,1)$.
\item $p \equiv 11 \pmod{16}$ and $(i,j)=(1,1)$.
\item $p \equiv  3 \pmod{8}$ and $(i,j)=(0,2)$.
\item $p \equiv - 3 \pmod{8}$ and $(i,j)=(0,2)$.
\end{enumerate}
Then, we have $\rank(J^{(p; i, j)}(\mathbb{Q})) = 0$.
\end{theorem}
We treat the above four cases separately but in a similar manner in the following four subsections respectively.
Since $J^{(p; i, j)}(\mathbb{Q)}/2J^{(p; i, j)}(\mathbb{Q})$ can be embedded into the $2$-Selmer group $\Sel(\mathbb{Q},J^{(p; i, j)})$, 
in order to bound the Mordell-Weil rank from above, it is sufficient to calculate the dimension of the $2$-Selmer group.
By \cite[p. 256]{Stoll2001}, we have
\begin{align*}
\dim \Sel(\mathbb{Q},J^{(p; i, j)}) =& \dim {\val}^{-1}(G) + \dim(\mathrm{Im}(\delta_{2}) \times\mathrm{Im}(\delta_{p})) \\
&- \dim((\mathrm{Im}(\delta_{2}) \times\mathrm{Im}(\delta_{p}))+\res_S({\val}^{-1}(G))).
\end{align*} 
In each case, we can prove that the right hand side equals $2$. 
\fn{In other cases, computation by MAGMA \cite{Bosma-Cannon-Playoust} suggests that $\dim_{\mathbb{F}_2} \Sel(\mathbb{Q}, J^{(p; i, j)})> 2$.
} 
\fn{Note that $\mathrm{Im}(\delta_{\infty})$ is trivial.
Note also that $\dim_{\mathbb{F}_2} \mathrm{Im}(\delta_{2})$
(resp. $\dim_{\mathbb{F}_2} \mathrm{Im}(\delta_{p})$) is independent of prime number $p$.
For the detail, see \cref{Jp01,J201,Jp11,J211,Jp02,J202,Jp02',J202'}.
More strongly, $C^{(p;i,j)}$ and $C^{(q;i,j)}$ are isomorphic over $\mathbb{Q}_2$ whenever $p \equiv q \pmod{16}$.  
Moreover, $C^{(p;0,2)}$ and $C^{(q;0,2)}$ are isomorphic over $\mathbb{Q}_2$ whenever $p^2 \equiv q^2 \pmod{16}$.
} 
Here and after, we follow the notation in \cite{Stoll2001} as below.

\begin{notation}
We fix $p$, $i$, and $j$, so we abbreviate $C^{(p;i,j)}$ to $C$ and $J^{(p;i,j)}$ to $J$.
Denote
\begin{itemize}
\item the $x$-coordinate of the point $P \in C(\mathbb{Q})$ by $x_{P}$,

\item every divisor class in $J(\mathbb{Q})$ represented by a divisor $D$ simply by $D$,

\item a fixed algebraic closure of $\mathbb{Q}$ by $\ol{\mathbb{Q}}$.

\end{itemize}
For every place $v$, we also use a similar notation and fix an embedding $\ol{\mathbb{Q}} \hookrightarrow \ol{\mathbb{Q}_v}$. 
Define
\begin{itemize}
\item $L:=\mathbb{Q}[T]/(f(T)) \overset{\simeq}{\to} \prod_{i=1}^3 L^{(i)}; \; T \mapsto (T_1;T_2;T_3)$,
 where $L^{(1)}:=\mathbb{Q}[T_1]/(T_1)$, $L^{(2)}:=\mathbb{Q}[T_2]/(T_2^2+2^ip^j)$ 
and $L^{(3)}:=\mathbb{Q}[T_3]/(T_3^2+2^{i+1}p^j)$.
We denote the trivial elements in $L^{\times}$ and $L^{\times}/{L^{\times 2}}$ by $\1$.
\end{itemize}
For every place $v$, define  
\begin{itemize}
\item
$L_v:=\mathbb{Q}_v[T]/(f(T)) \overset{\simeq}{\to}  \prod_{i=1}^m L_v^{(i)};$ $T \mapsto (T_1;\cdots;T_m)$, where $f(T)=f_1(T) \cdots f_m(T)$ is the irreducible decomposition in $\mathbb{Q}_v[T]$ and $L_v^{(i)}:=\mathbb{Q}_v[T_i]/(f_i(T_i))$.
We denote the trivial elements in $L_v^{\times}$ and $L_v^{\times}/{L_v^{\times 2}}$ by $\1_v$.

\item
$\delta_{v} : J(\mathbb{Q}_{v}) \to L_v^{\times}/L_v^{\times 2};$ $D = \sum_{i = 1}^{n} m_{i} P_{i} \mapsto  \prod_{i = 1}^{n}(x_{P_{i}}-T)^{m_{i}}$, where $D$ is a divisor whose support is disjoint from the support of the divisor $\mathrm{div}(y)$

\item
$\res_v: L^{\times}/L^{\times 2} \to L_v^{\times}/L_v^{\times 2}$ as the map induced by $L \to L_v;$ $T \mapsto T$.

\item
$I_v(L):=\prod_{i=1}^3 \mbox{(the group of fractional ideals of $L_v^{(i)}$)}$.
We often denote
\begin{itemize}
\item 
 each element of $I_v(L)$ and $I_v(L)/I_v(L)^2$ by $(\mathfrak{a}^{(1)}; \ldots; \mathfrak{a}^{(m)})_v$. 
 \item  the trivial elements of $I_v(L)$ and $I_v(L)/I_v(L)^2$ by $\1_v$.
\end{itemize}
\item
$\val_v: L_v^{\times} \to I_v(L);$ $a \mapsto (a)$. 
\fn{Let $\mathfrak{q}_{v}^{(i, 1)}, ..., \mathfrak{q}_{v}^{(i, j_{i})}$ be the prime ideals of $L^{(i)}$ above $v$, then
\begin{align*}
\val_v(\alpha_{1}, \dots, \alpha_{m})
= \left(\alpha_{1}^{v\left(\mathfrak{q}_{v}^{(1, j_{1})} \right)}, \dots, \alpha_{m}^{v\left( \mathfrak{q}_{v}^{(1, j_{1})} \right)}; \dots;  \alpha_{1}^{v\left( \mathfrak{q}_{v}^{(m, 1)} \right)}, \dots, \alpha_{m}^{v\left( \mathfrak{q}_{v}^{(m, j_{m})} \right)} \right),
\end{align*}
which is well-defined up to the order of $\mathfrak{q}_{v}^{(i, 1)}, ..., \mathfrak{q}_{v}^{(i, j_{i})}$ for each $i$. 
Here, we identify the prime number $v$ and the associated valuation map.}
We denote the induced map $L_v^{\times}/L_v^{\times 2} \to I_v(L)/I_v(L)^2$ simply by $\val_v$.

\item 
$G_{v} := \val_{v}(\mathrm{Im}(\delta_{v})) \subset I_v(L)/I_v(L)^2$.
\end{itemize}
Finally, define
\begin{itemize}
\item
$I(L) :=\prod_v I_v(L) \simeq \prod_{i=1}^3 (\mbox{the group of fractional ideals of $ L^{(i)}$})$.
We identify $I_2(L) \times I_p(L)$ with its image in $I(L)$.

\item $\Cl(L):=\prod_{i=1}^3 \Cl(L^{(i)})$, where $\Cl(L^{(i)})$ are the ideal class groups of $L^{(i)}$. 

\item $S := \{2, p, \infty\}$.

\item 
$\res_S:=\prod_{v \in S \setminus \{\infty\}} \res_v: L^{\times}/L^{\times 2} \to \prod_{v \in S \setminus \{\infty\}} L_v^{\times}/L_v^{\times 2}$.

\item $\val:=\prod_{v} \val_v \circ \res_v: L^{\times} \to I(L)$. 
We denote the induced map $L^{\times}/L^{\times 2} \to I(L)/I(L)^2$ simply by $\val$.

\item
$G:=\prod_{v \in S \setminus \{\infty\}} G_v \subset I(L)/I(L)^2$.

\item
$W:=\Ker(G \to \Cl(L)/\Cl(L)^2)$.
\end{itemize}
\end{notation}


\subsection{Case $(i,j)=(0,1)$}

Suppose that $p \equiv 3 \pmod{16}$. Then, we have the following irreducible decompositions:
\[\begin{cases}
L_2=\mathbb{Q}_2[T_1]/(T_1) \times \mathbb{Q}_2[T_2]/(T_2^2+p) \times \mathbb{Q}_2[T_3]/(T_3^2+2p), \\
L_p=\mathbb{Q}_p[T_1]/(T_1) \times \mathbb{Q}_p[T_2]/(T_2^2+p) \times \mathbb{Q}_p[T_3]/(T_3^2+2p).
\end{cases}\] 

\begin{lemma} \label{2 torsion01}
The following two elements form an $\mathbb{F}_{2}$-basis of $J(\mathbb{Q})[2]$ and $J(\mathbb{Q}_{v})[2]$ for $v = 2, p, \infty$: 
\begin{align*}
(0,0) &- \infty ,&
\sum_{\substack{P \in C(\ol{\mathbb{Q}}) \\ x_{P}^{2}+p = 0}}P &- 2\infty. 
\end{align*}

In particular, we have the following table.
\fn{The condition $p \equiv 3 \pmod{16}$ is used only to deduce that $-p$ is not square in $\mathbb{Q}_{v}$ for $v = 2$, $p$. In fact, this lemma is true for $p \equiv 1, 3, 5  \pmod{8}$.}
\begin{table}[ht]
    \begin{tabular}{|l||l|l|} \hline
  $v$   & $\dim J(\mathbb{Q}_{v})[2]$ & $\dim \mathrm{Im}(\delta_{v})$ 
  \\ \hline \hline
     $2$   & $2$ & $4$      \\ \hline
    $p$   & $2$   & $2$       \\ \hline
      $\infty$   & $2$ & $0$   \\ \hline  
    \end{tabular}
\end{table}
\end{lemma}

\begin{proof}
The first statement follows from \cite[Lemma 5.2]{Stoll2014}. 
Note that neither $-p$ nor $-2p$ is not square in $\mathbb{Q}_{v}$ for $v = 2, p, \infty$. 
The second statement follows from the following formula (cf. \cite[p. 451, proof of Lemma 3]{FPS}).
\[\dim_{\mathbb{F}_2} \mathrm{Im}(\delta_{v}) =\dim_{\mathbb{F}_2} J({\mathbb{Q}_{v}})[2]
\begin{cases}
+0 & (v \neq 2, \infty), \\
+2 & (v = 2), \\
-2 & (v = \infty). \\
\end{cases}\]
\end{proof}

In the calculation of $\mathrm{Im}(\delta_{v})$, we use the following formula.
\begin{lemma} [{\cite[Lemma 2.2]{Schaefer}}]  \label{image of 2-torsion}
Let $C:y^2=f(x)$ be a hyperelliptic curve over $\mathbb{Q}$ such that $\deg f$ is odd. 
For every place $v$ of $\mathbb{Q}$, any point on $J(\mathbb{Q}_v)$ can be represented by a divisor of degree $0$ whose support is disjoint from the support of the divisor $\mathrm{div}(y)$.
 Then, we have $\delta_{v}(D) = 1$ if $D$ is supported at $\infty$.
 If $D$ is of the form $D=\sum_{i=1}^{n} D_i$ with $D_i=(\alpha_i,0)$, where 
 $\alpha_i$ runs through all roots of a monic irreducible factor $h(x) \in \mathbb{Q}_v[x]$ of $f(x)$, then we have
 \[\delta_{v}(D)=(-1)^{\deg h} \left(h(x)-\frac{f(x)}{h(x)} \right).\] 
\end{lemma} 
For example, we obtain the following lemma.

\begin{lemma} 
\label{Jp01} 
The following two elements of $L_p^{\times}/L_p^{\times 2}$ form an $\mathbb{F}_{2}$-basis of $\mathrm{Im}(\delta_{p})$:
\begin{align*}
(2; -T_2; -T_3),&& (p; T_2; 2). 
\end{align*}
In particular, the following two elements of $I_p(L)/I_p(L)^2$ form an $\mathbb{F}_{2}$-basis of $G_p$: 
\begin{align*}
((1); (T_2); (p, T_3))_p,&& ((p); (T_2); (1))_p.
\end{align*}
\end{lemma}
\begin{proof}
\cref{image of 2-torsion} implies that
\begin{align*}
& \delta_p ((0,0)-\infty)=-T+(T^2+p)(T^2+2p)
= (2;-T_2;-T_3),\\
& \delta_p \left(\sum_{\substack{P \in C(\ol{\mathbb{Q}_p}) \\ x_{P}^{2}+p = 0}}P - 2\infty \right)=(T^2+p)-T(T^2+2p)
= (p;T_2;2).
\end{align*}
Hence, $(2; -T_2; -T_3), (p; T_2; 2) \in \mathrm{Im}(\delta_{p})$.

By \cref{2 torsion01}, it is sufficient to prove that the above two elements are linearly independent.
Since $v_{T_2}(-T_2)=v_{T_2}(T_2)=1$, $(2; -T_2; -T_3)$ and $(p; T_2; 2)$ are non-trivial in $L^{\times}/L^{\times 2}$.
Moreover, since $v_p(2p)=1$, $(2; -T_2; -T_3)(p; T_2; 2)=(2p;-1;-2T_3)$ is non-trivial in $L^{\times}/L^{\times 2}$.
Thus, the lemma holds.
\end{proof}

\begin{lemma} 
\label{J201} 
\begin{enumerate} 

\item Suppose that $p \equiv 3 \pmod{32}$. Then, the following four elements of $L_2^{\times}/L_2^{\times 2}$ form an $\mathbb{F}_{2}$-basis of $\mathrm{Im}(\delta_{2})$:
 \begin{align*}
 (2;-T_2;-T_3), &&  (3;T_2;-3),  && (-1;-1-T_2;-1-T_3), && (-1;4+2T_2;1+2T_3). 
 \end{align*}
\item Suppose that $p \equiv 19 \pmod{32}$. Then, the following four elements  of $L_2^{\times}/L_2^{\times 2}$ form an $\mathbb{F}_{2}$-basis of $\mathrm{Im}(\delta_{2})$:
\begin{align*}
(2;-T_2;-T_3), && (3;T_2;-3), && (-3;5-T_2;5-T_3), && (-1;4+2T_2;1+2T_3). 
 \end{align*}
 In particular, the following two elements of $I_2(L)/I_2(L)^2$ form an $\mathbb{F}_{2}$-basis of $G_2$: 
\begin{align*}
((2);(1);(2,T_3))_2,&& ((1);(2);(1))_2.
 \end{align*}
\end{enumerate}
\end{lemma}
 
\begin{proof}
\begin{enumerate}
\item
First, we show that the above four elements actually lie in $\mathrm{Im}(\delta_{2})$. Indeed,
\cref{image of 2-torsion} implies that
 \begin{align*}
& \delta_2 ((0,0)-\infty) =-T+(T^2+p)(T^2+2p) = (2;-T_2;-T_3),\\
 & \delta_2 \left(\sum_{\substack{P \in C(\ol{\mathbb{Q}_2}) \\ x_{P}^{2}+p = 0}}P-2\infty \right) =(T^2+p)-T(T^2+2p)=(3;T_2;-3).\\
    \end{align*} 
Hence, $(2;-T_2;-T_3)$, $(3;T_2;-3) \in \mathrm{Im}(\delta_{2})$.
Moreover, since $f(-1) \equiv 4 \pmod{32}$, there exists $Q \in C(\mathbb{Q}_2)$ such that $x_Q=-1$, and \[(-1;-1-T_2;-1-T_3)= \delta_2 (Q-\infty)
 \in \mathrm{Im}(\delta_{2}).\]
Finally, since $2=\tau^2$ in $\mathbb{Q}_2[\tau]/(\tau^2-2)$ and $p \equiv 1+\tau^2 \pmod{{\tau}^7}$, 
we can check that  \[f(1+\tau+\tau^2) \equiv (\tau+\tau^2+\tau^3)^2 \pmod{\tau^7}, 
\]
hence $f(1+\tau+\tau^2)$ is square in $\mathbb{Q}_2[\tau]/(\tau^2-2)$ by using Hensel's lemma. 
Therefore, we see that there exist $R$, $\ol{R} \in C(\ol{\mathbb{Q}_2})$ such that $x_R+x_{\ol{R}}=6$, $x_Rx_{\ol{R}}=7$, and
$R+\ol{R}-2 \infty$ defines an element of $J(\mathbb{Q}_2)$. 
\fn{We can take each element of $J(\mathbb{Q}_2)$ in one of the seven quadratic extensions of $\mathbb{Q}_2$.
} 
Thus,
\begin{align*}
 (-1;4+2T_2; 1+2T_3)=\delta_2(R+\ol{R}-2 \infty) \in \mathrm{Im}(\delta_{2}).
\end{align*}

By \cref{2 torsion01}, it is sufficient to prove that the above four elements are linearly independent.
By taking their first components into account, we see that they are non-trivial in $L^{\times}/L^{\times 2}$,
and the only possible relation is \[(-1;-1-T_2;-1-T_3)  (-1;4+2T_2;1+2T_3)=\1.\]
However, this is impossible because \[(-1-T_3)(1+2T_3) \equiv 1+T_3 \not\equiv 1 \pmod{T_3^2}.\]
Therefore, we obtain the assertion.

\item
The proof is similar to $(1)$. 
\cref{image of 2-torsion} implies that $(2;-T_2;-T_3)$, $(3;T_2;-3) \in \mathrm{Im}(\delta_{2})$. 
Since $f(5) \equiv 4 \pmod{32}$, there exists $Q \in C(\mathbb{Q}_2)$ such that $x_Q=5$, and $ (-3;5-T_2;5-T_3)= \delta_2 (Q-\infty)
 \in \mathrm{Im}(\delta_{2})$. Moreover, $(-1;4+2T_2;1+2T_3) \in \mathrm{Im}(\delta_{2})$ by the proof of $(1)$.

By \cref{2 torsion01}, it is sufficient to prove that the above four elements are linearly independent.
By taking their first components into account, they are non-trivial in $L^{\times}/L^{\times 2}$, and only possible relation is \[(3;T_2;-3) (-3;5-T_2;5-T_3)  (-1;4+2T_2;1+2T_3)=\1.\]  
However, it is impossible because
\[-p(5-T_3)(1+2T_3) \equiv 1+T_3 \not\equiv 1  \pmod{T_3^2}.\]
Therefore, we obtain the assertion.
\end{enumerate}

By the first statement, to prove the second, it is sufficient to show that
\[\begin{cases}
\val_2(2;-T_2;-T_3)=((2);(1);(2,T_3))_2,\\
\val_2(3;T_2;-3)=\1_2,\\
\val_2(-1;-1-T_2;-1-T_3)=((1);(2);(1))_2,\\
\val_2(-1;4+2T_2;1+2T_3)=((1);(2);(1))_2,\\
\val_2(-3;5-T_2;5-T_3)=((1);(2);(1))_2.
\end{cases} \]
The first components are obvious. 
Since $p \equiv 3 \pmod{8}$, we can calculate the second components as follows.
\begin{itemize}
\item $v_2(T_2)=v_2(-T_2)=0$.
\item  Since $v_2(N(-1-T_2))=v_2(4)=2$ and $v_2(\tr(-1-T_2))=v_2(-2)=1$, we have $v_2(-1-T_2)=1$.
\item Since $v_2(N(4+2T_2))=v_2(28)=2$ and $v_2(\tr(4+2T_2))=v_2(8)=3$, we have $v_2(4+2T_2)=1$.
\item Since $v_2(N(5-T_2))=v_2(28)=2$ and $v_2(\tr(5-T_2))=v_2(10)=1$, we have $v_2(5-T_2)=1$.
\end{itemize}
Since $T_3$ is a uniformizer in $L^{(3)}_2$, we obtain the third components.
\end{proof}

\begin{remark} \label{two cases}
In fact, we can relate the above two cases in a more direct manner as follows:
Let 
\begin{align*}
M_2 :=\mathbb{Q}_2[\sigma]/(\sigma(\sigma^2+3)(\sigma^2+6)), && L_2:=\mathbb{Q}_2[\tau]/(\tau(\tau^2+19)(\tau^2+38)).
\end{align*}
Then, we have an isomorphism
$M_2 \overset{\simeq}{\to} L_2; \sigma \mapsto \tau/u^2$ of $\mathbb{Q}_2$-algebras, where $u \in \mathbb{Z}_2^{\times}$ is taken so that $19=3u^4$.
It induces an isomorphism $M_2^{\times}/M_2^{\times 2} \overset{\simeq}{\to} L_2^{\times}/L_2^{\times 2}$ which maps $(-1;-1-\sigma_2;-1-\sigma_3)$ to $-9-\tau=(3;\tau_2;-3)(-3;5-\tau_2;5-\tau_3)$, where
$\sigma_2$, $\sigma_3$ (resp. $\tau_2$, $\tau_3$) are the images of $\sigma$ (resp. $\tau$) in 
$\mathbb{Q}_2[\sigma]/(\sigma^2+3)$, $\mathbb{Q}_2[\sigma]/(\sigma^2+6)$ (resp. $\mathbb{Q}_2[\tau]/(\tau^2+19)$, $\mathbb{Q}_2[\tau]/(\tau^2+38)$).
\end{remark}

The following proposition is essentially due to Gauss' genus theory. 

\begin{proposition} [{cf. \cite[Theorem $2$]{Morton}, \cite[Proposition 3.11, Theorem 5.30]{Cox}}] \label{Genus theory} 
Let $K$ be an imaginary quadratic field of discriminant $D$.
\begin{enumerate}
\item The class of an ideal $\mathfrak{a}$ of $K$ lies in $\Cl(K)^2$ if and only if the equation $N(\mathfrak{a})x^2+Dy^2=z^2$ has a primitive $\mathbb{Z}_l$-solution 
for every prime number $l$ dividing $D$.
\item Let $r$ be the number of odd prime numbers dividing $D$.
Then, we have
\[ \# \Cl(K)[2]=\begin{cases}
2^{r-1} & D \equiv 1 \pmod{4}, \\
2^r & D \equiv 0 \pmod{4}. 
\end{cases}\]
\end{enumerate}
\end{proposition}

\begin{corollary} \label{2Cl}
The class of $(2,T_3)$ does not lie in $\Cl(L^{(3)})^2$. 
In particular, $\Cl(L^{(3)})[2]$ is generated by $(2,T_3)$.
\end{corollary}

\begin{proof}
By \cref{Genus theory} $(1)$, it is sufficient to prove that $2x^2-8py^2=z^2$ has no primitive $\mathbb{Z}_2$-solution.
We prove it by contradiction.

Suppose that $2x^2-8py^2=z^2$  has a primitive $\mathbb{Z}_2$-solution $(x,y,z)$. Then, we have
\begin{itemize}
\item $x^2-4py^2=2z'^2$ for some $z' \in \mathbb{Z}_2$, hence
\item $2x'^2-2py^2=z'^2$ for some $x' \in \mathbb{Z}_2$, hence 
\item $x'^2-py^2=2z''^2$ for some $z'' \in \mathbb{Z}_2$.
\end{itemize}
Since $p \equiv 3 \pmod{8}$ and $y \in \mathbb{Z}_2^{\times}$, $x'^2-py^2$ is congruent to $1$, $5$ or $6 \pmod{8}$. However, $2z''^2$ 
is congruent to $0$ or $2 \pmod{8}$, a contradiction.
Therefore, the first statement follows. The second statement follows from \cref{Genus theory} $(2)$ and $(2,T_3)^2=(2)$ as ideals.
\end{proof}

\begin{corollary} \label{V01}
The following four elements of $L^{\times}/ L^{\times 2}$ form an $\mathbb{F}_{2}$-basis of $\Ker(\val)$: 
 \begin{align*}
 (1;1;-1), && (1;-1;1), && (-1;1;1), && (1;1;2).
 \end{align*}
\end{corollary}

\begin{proof}
We have the following commutative diagram.
 \[ \xymatrix@R+2em@C+2em{
 & \mathcal{O}_L^{\times} \ar[r]^{2}
  \ar[d] &   \mathcal{O}_L^{\times}  \ar[r] \ar[d]  & \Ker(\val) \ar[d] & \\
& L^{\times} \ar[r]^{2} \ar[d]^{\val} &  L^{\times} \ar[r] \ar[d]^{\val}  & L^{\times}/L^{\times 2}  \ar[d]^{\val}
 \ar[r] & 0\\
0 \ar[r] & I(L) \ar[r]^{2} \ar[d] &  I(L) \ar[r] \ar[d]  & I(L)/{I(L)}^{2}  \ar[r]& 0\\
&  \Cl(L)  \ar[r]^{2} & \Cl(L)  &&} \]
By applying the snake lemma, we obtain a short exact sequence 
\[1 \to \mathcal{O}_L^{\times}/\mathcal{O}_L^{\times 2} \to \Ker(\val) \to \Cl(L)[2] \to 0.\] 
First, $(1;1;-1)$, $(1;-1;1)$ and $(-1;1;1)$ form an $\mathbb{F}_2$-basis of $\mathcal{O}_L^{\times}/\mathcal{O}_L^{\times 2}$.
Since the ring of integers of $L^{(1)} \simeq \mathbb{Q}$ is a PID, we have $\Cl(L^{(1)})[2]=0$.
By \cref{Genus theory} $(2)$, we have $\Cl(L^{(2)})[2]=0$. 
\cref{2Cl} implies that $\Cl(L^{(3)})[2]$ is generated by the class of $(2,T_3)$.
Therefore, by a diagram chasing, we can check that the image of $(1;1;2)$ in $\Cl(L)[2] $ is non-trivial.
This completes the proof. 
\end{proof}

\begin{lemma} \label{W01} 
The following three elements of $I(L)/I(L)^2$ form an $\mathbb{F}_{2}$-basis of $W$: 
\begin{align*}
((1);(2);(1))_2 \times \1_p, &&
\1_2 \times ((p);(T_2);(1))_p, &&
((2);(1);(2,T_3))_2 \times ((1);(T_2);(p,T_3))_p.
\end{align*}
\end{lemma}

\begin{proof}
By \cref{Jp01,J201}, the following four elements of $I(L)/I(L)^2$ form an $\mathbb{F}_{2}$-basis of $G$: 
\begin{align*}
 ((1);(2);(1))_2 \times \1_p, && ((2);(1);(2,T_3))_2 \times \1_p,&&  \1_2 \times ((p);(T_2);(1))_p,&& \1_2 \times ((1);(T_2);(p,T_3))_p.
 \end{align*}
Since the images of 
\begin{align*}
((1);(2);(1))_2 \times \1_p,&& \1_2 \times ((p);(T_2);(1))_p,&& ((2);(1);(2,T_3))_2 \times ((1);(T_2);(p,T_3))_p 
\end{align*}
in $\Cl(L)$ are trivial, the three elements in the statement lie in $W$. 
Moreover, we see that they are linearly independent.
Finally, by \cref{2Cl}, we have
$((2);(1);(2,T_3))_2 \times \1_p \not\in W$.
Therefore, we obtain the assertion.
\end{proof}

\begin{lemma} \label{H01}
The following seven elements of $L^{\times}/L^{\times 2}$ form an $\mathbb{F}_{2}$-basis of ${\val}^{-1}(G)$: 
\begin{align*}
(1;1;-1), && (1;-1;1),&& (-1;1;1),&& (1;1;2), && (1;2;1), && (p;T_2;1),&& (2;T_2;T_3) .
\end{align*}
\end{lemma}
\begin{proof}
By the definition of $W$, 
we have $\val(\val^{-1}(G)) \subset W$.
Moreover, we have $W \subset \val(\val^{-1}(G))$:
Indeed, for every $[\mathfrak{a}] \in W$, there exists $\mathfrak{b} \in I(L)$ such that $[\mathfrak{a}\mathfrak{b}^2]=[(1)]$ in $\Cl(L)/\Cl(L)^2$.
Thus, there exists $a \in L^{\times}$ such that $[\mathfrak{a}]=[\mathfrak{a}\mathfrak{b}^2]=[(a)]$ in $I(L)/I(L)^2$.
Therefore, we obtain the following short exact sequence:
\[0 \to \Ker(\val) \to \val^{-1}(G) \overset{\val}{\to} W \to 0.\]
By \cref{V01,W01}, we obtain the assertion.
\end{proof}

\begin{lemma} \label{H+JS01}
Set elements of $L_2^{\times}/L_2^{\times 2} \times L_p^{\times}/L_p^{\times 2}$ as follows: 
\begin{align*}
d_1 &:= (2;-T_2;-T_3)_2 \times\1_p, \\
d_2 &:= (3;T_2;-3)_2 \times\1_p, \\
d_3 &:= (-1;-1-T_2;-1-T_3)_2 \times\1_p,   & \tilde{d_3}&:=(-3;5-T_2;5-T_3)_2 \times\1_p,\\
d_4 &:=  (-1;4+2T_2;1+2T_3)_2 \times\1_p, \\
d_5 &:= \1_2 \times(2;-T_2;-T_3)_p,\\
d_6 &:=\1_2 \times(p;T_2;2)_p,\\
h_1 &:= \res_S(1;1;-1)=(1;1;-1)_2 \times(1;1;-1)_p, \\
h_2 &:= \res_S(1;-1;1)=(1;-1;1)_2 \times(1;-1;1)_p, \\
h_3 &:= \res_S(-1;1;1)=(-1;1;1)_2 \times(-1;1;1)_p, \\
h_4 &:= \res_S(1;1;2)=(1;1;2)_2 \times(1;1;2)_p, \\
h_5 &:= \res_S(1;2;1)=(1;2;1)_2 \times(1;2;1)_p, \\
h_6 &:=\res_S(p;T_2;1)=(p;T_2;1)_2 \times(p;T_2;1)_p, \\
h_7 &:= \res_S(2;T_2;T_3)=(2;T_2;T_3)_2 \times(2;T_2;T_3)_p.
\end{align*}
\begin{enumerate}
\item We have
$d_1 d_5 h_1 h_2 h_7 =d_2 d_6 h_4 h_6 = \1_2 \times \1_p.$
\item  Suppose that $p \equiv 3 \pmod{32}$. 
Then, $d_1$, $d_2$, $d_3$, $d_4$, $h_1, \ldots, h_7$ form an $\mathbb{F}_2$-basis of $(\mathrm{Im}(\delta_{2}) \times\mathrm{Im}(\delta_{p}))+\res_S({\val}^{-1}(G))$.
\item Suppose that $p \equiv 19 \pmod{32}$. 
Then, $d_1$, $d_2$, $\tilde{d_3}$, $d_4$, $h_1, \ldots, h_7$ form an $\mathbb{F}_2$-basis of $(\mathrm{Im}(\delta_{2}) \times\mathrm{Im}(\delta_{p}))+\res_S({\val}^{-1}(G))$.
\end{enumerate}
\end{lemma}

\begin{proof}
\begin{enumerate} \item We can check it by direct calculation.

\item By $(1)$, \cref{Jp01,J201,H01}, 
$(\mathrm{Im}(\delta_{2}) \times\mathrm{Im}(\delta_{p}))+\res_S({\val}^{-1}(G))$ is generated by $d_1, \ldots d_4, h_1,  \ldots , h_7$. 
Set
\begin{align*}
d_1&=(2;-T_2;-T_3)_2 \times\1_p, \\
d_2&=(3;T_2;-3)_2 \times\1_p, \\
d_3' &:= d_3d_4=(1;(-1-T_2)(4+2T_2);(-1-T_3)(1+2T_3))_2 \times\1_p, \\ 
d_4&= (-1;4+2T_2;1+2T_3)_2 \times\1_p, \\
h_1' &:= h_1h_4=(1;1;-2)_2 \times\1_p,\\
h_2' &:= h_2h_5 = (1;-2;1)_2 \times\1_p, \\
h_3' &:= h_3 h_7 =(-2;T_2;T_3)_2 \times(1;T_2;T_3)_p,\\
h_4&=\res_S(1;1;2)=(1;1;2)_2 \times(1;1;2)_p, \\
h_5&=\res_S(1;2;1)=(1;2;1)_2 \times(1;2;1)_p, \\
h_6&=\res_S(p;T_2;1)=(p;T_2;1)_2 \times(p;T_2;1)_p, \\
h_7&=\res_S(2;T_2;T_3)=(2;T_2;T_3)_2 \times(2;T_2;T_3)_p.
\end{align*}
It is sufficient to prove that the above eleven elements are linearly independent.
\begin{enumerate}
\item By taking the first components at $v=p$ into account, we see that there is no relation containing $h_6$ and $h_7$ . 
\item By taking the second components at $v=p$ into account, we see that there is no relation containing $h_3'$ and $h_5$. 
\item By taking the third components at $v=p$ into account, we see that there is no relation containing $h_4$.
\item By taking the first components at $v=2$ into account, we see that there is no relation containing $d_1$,  $d_2$ and $d_4$.
\item By taking the third components at $v=2$ into account, we see that there is no relation containing $d_3'$ and $h_1'$ because
\begin{itemize}
\item 
$(-1-T_3)(1+2T_3) \equiv 1+T_3 \pmod{T_3^2}$,
\item 
$-2 \equiv T_3^2(1+T_3^2) \pmod{T_3^6}$,
\item 
$-2(-1-T_3)(1+2T_3) \equiv T_3^2(1+T_3) \pmod{T_3^4}$.
\end{itemize}
\item Finally, by taking the second components at $v=2$ into account, we see that there is no relation containing $h_2'$.
\end{enumerate}

\item Respectively, $(2)$ holds.

By $(1)$, \cref{Jp01,J201,H01}, $(\mathrm{Im}(\delta_{2}) \times\mathrm{Im}(\delta_{p}))+\res_S({\val}^{-1}(G))$ is generated by $d_1, \ldots d_4, h_1,  \ldots , h_7$. 
Set
\begin{align*}
d_1&=(2;-T_2;-T_3)_2 \times\1_p, \\
d_2'&:= d_2d_3d_4=(1;T_2(5-T_2)(4+2T_2);-p(5-T_3)(1+2T_3))_2 \times\1_p, \\
\tilde{d_3} &= (-3;5-T_2;5-T_3)_2 \times\1_p, \\ 
d_4&= (-1;4+2T_2;1+2T_3)_2 \times\1_p, \\
h_1' &:= h_1h_4=(1;1;-2)_2 \times\1_p,\\
h_2' &:= h_2h_5 = (1;-2;1)_2 \times\1_p, \\
h_3' &:= h_3 h_7 =(-2;T_2;T_3)_2 \times(1;T_2;T_3)_p,\\
h_4&=\res_S(1;1;2)=(1;1;2)_2 \times(1;1;2)_p, \\
h_5&=\res_S(1;2;1)=(1;2;1)_2 \times(1;2;1)_p, \\
h_6&=\res_S(p;T_2;1)=(p;T_2;1)_2 \times(p;T_2;1)_p, \\
h_7&=\res_S(2;T_2;T_3)=(2;T_2;T_3)_2 \times(2;T_2;T_3)_p.
\end{align*}
It is sufficient to prove that the above eleven elements are linearly independent.

\begin{enumerate}
\item By taking the first components at $v=p$ into account, we see that there is no relation containing $h_6$ and $h_7$. 
\item By taking the second components at $v=p$ into account, we see that there is no relation containing $h_3'$ and $h_5$.
\item By taking the third components at $v=p$ into account, we see that there is no relation containing $h_4$.
\item By taking the first components at $v=2$ into account, we see that there is no relation containing $d_1$, $\tilde{d_3} $ and $d_4$.
\item By taking the third components at $v=2$ into account, we see that there is no relation containing $d_2'$ and $h_1'$ because
\begin{itemize}
\item $-p(5-T_3)(1+2T_3) \equiv 1+T_3 \pmod{T_3^2}$, 
\item 
$-2  \equiv T_3^2(1+T_3^2) \pmod{T_3^6}$, 
\item Since $2p(5-T_3)(1+2T_3) \equiv T_3^2(1+T_3) \pmod{T_3^4}$.
\end{itemize}
\item Finally, by taking the second components at $v=2$ into account, we see that there is no relation containing $h_2'$.
\end{enumerate}
\end{enumerate}
\end{proof}
Recall that 
\begin{align*}
\dim \Sel(\mathbb{Q},J^{(p; i, j)}) =& \dim {\val}^{-1}(G) + \dim(\mathrm{Im}(\delta_{2}) \times\mathrm{Im}(\delta_{p})) \\
&- \dim((\mathrm{Im}(\delta_{2}) \times\mathrm{Im}(\delta_{p}))+\res_S({\val}^{-1}(G)))
\end{align*}
and \[J(\mathbb{Q})/2J(\mathbb{Q}) \simeq \mathbb{F}_2^{\rank(J(\mathbb{Q}))} \oplus J(\mathbb{Q})[2].\]
Therefore, by \cref{2 torsion01,Jp01,J201,H01,H+JS01}, we obtain
\begin{align*}
\rank(J(\mathbb{Q}))  \leq \dim \Sel(\mathbb{Q},J)-\dim J(\mathbb{Q})[2]=7+6 -11-2=0. 
\end{align*}
This completes the proof of \cref{rank=0} $(1)$.

\subsection{Case $(i,j)=(1,1)$}

Suppose that $p \equiv 11 \pmod{16}$. Then, we have the following irreducible decompositions:
\[\begin{cases}
L_2=\mathbb{Q}_2[T_1]/(T_1) \times \mathbb{Q}_2[T_2]/(T_2^2+2p) \times \mathbb{Q}_2[T_3]/(T_3^2+4p), \\
L_p=\mathbb{Q}_p[T_1]/(T_1) \times \mathbb{Q}_p[T_2]/(T_2^2+2p) \times \mathbb{Q}_p[T_3]/(T_3^2+4p).
\end{cases}\] 
\begin{lemma} \label{2 torsion11}
The following two elements form an $\mathbb{F}_{2}$-basis of $J(\mathbb{Q})[2]$ and $J(\mathbb{Q}_{v})[2]$ for $v = 2, p, \infty$: 
\begin{align*} 
(0,0) &- \infty, & \sum_{\substack{P \in C(\ol{\mathbb{Q}}) \\ x_{P}^{2}+2p = 0}}P &- 2\infty.
\end{align*}
In particular, we have the following table.

\begin{table}[ht]
    \begin{tabular}{|l||l|l|} \hline
  $l$   & $\dim J(\mathbb{Q}_v)[2]$ & $\dim \delta_{v}(J(\mathbb{Q}_v))$ 
  \\ \hline \hline
     $2$   & $2$ & $4$
     \\ \hline
     $p$   & $2$  & $2$ 
   \\ \hline
      $\infty$   & $2$ & $0$ 
      \\ \hline  
          \end{tabular}
\end{table}
\end{lemma}

\begin{proof} 
The first statement follows from \cite[Lemma 5.2]{Stoll2014}.
Note that neither $-p$ nor $-2p$ is not square in $\mathbb{Q}_v$ for $v=2,p, \infty$. 
The second statement follows from the following formula (cf. \cite[p. 451, proof of Lemma 3]{FPS}).
\[\dim_{\mathbb{F}_2}  \delta_{v}(J(\mathbb{Q}_v)) =\dim_{\mathbb{F}_2} J({\mathbb{Q}_v})[2]
\begin{cases}
+0 & (v \neq 2, \infty), \\
+2 & (v=2),\\
-2 & (v=\infty).\\
\end{cases}\]
\end{proof}

\begin{lemma}
\label{Jp11} 
The following two elements of $L_p^{\times}/L_p^{\times 2}$ form an $\mathbb{F}_2$-basis of $\mathrm{Im}(\delta_{p})$:
\begin{align*}
(2;-T_2;-T_3),&& (2p;T_2;2). 
\end{align*}
 In particular, the following two elements of $I_p(L)/I_p(L)^2$ form an $\mathbb{F}_{2}$-basis of $G_p$: 
\begin{align*}
((1);(p,T_2); (T_3))_p, && ((p);(p,T_2);(1))_p.
\end{align*}
\end{lemma}
\begin{proof}
\cref{image of 2-torsion} implies that
\begin{align*}
&\delta_p ((0,0)-\infty)=-T+(T^2+2p)(T^2+4p)=(2;-T_2;-T_3), \\
&\delta_p \left(\sum_{\substack{P \in C(\ol{\mathbb{Q}_p}) \\ x_{P}^{2}+2p = 0}}P-2\infty \right)=(T^2+2p)-T(T^2+4p)=(2p;T_2;2).
\end{align*}
Hence, $(2;-T_2;-T_3), (2p;T_2;2) \in \mathrm{Im}(\delta_{p})$.

By \cref{2 torsion11}, it is sufficient to prove that the above two elements are linearly independent.
Since $v_{T_3}(-T_3)=1$ and $v_p(2p)=1$, $(2; -T_2; -T_3)$ and $(2p;T_2;2)$ are non-trivial in $L^{\times}/L^{\times 2}$.
Moreover, since $v_p(p)=1$, $(2;-T_2;-T_3) (2p;T_2;2)=(p;-1;-2T_3)$ is non-trivial in $L^{\times}/L^{\times 2}$.
Thus, the lemma holds.
\end{proof}

\begin{lemma} 
\label{J211} 
The following four elements of $L_2^{\times}/L_2^{\times 2}$ form an $\mathbb{F}_{2}$-basis of $\mathrm{Im}(\delta_{2})$: 
\begin{align*}
(2;-T_2;-T_3),&& (6;T_2;2),&& (3;3-T_2;3-T_3),&& (5;3;2-2T_3). 
\end{align*} 
In particular, the following two elements of $I_2(L)/I_2(L)^2$ form an $\mathbb{F}_{2}$-basis of $G_2$: 
\begin{align*}
((2);(2,T_2);(2))_2,&& ((1);(1);(2))_2. 
\end{align*}
\end{lemma}

\begin{proof}
First, we show that the four elements in the statement actually lie in $\mathrm{Im}(\delta_{2})$.
Indeed, \cref{image of 2-torsion} implies that
\begin{align*}
&\delta_2 ((0,0)-\infty)=-T+(T^2+2p)(T^2+4p)=(2;-T_2;-T_3),  \\
&\delta_2 \left(\sum_{\substack{P \in C(\ol{\mathbb{Q}}) \\ x_{P}^{2}+2p = 0}}P-2\infty \right)=(T^2+2p)-T(T^2+4p)=(6;T_2;2). 
\end{align*}
Hence, $(2;-T_2;-T_3), (6;T_2;2) \in \mathrm{Im}(\delta_{2})$. Moreover, since $f(3) \equiv 1 \pmod{8}$, there exists $Q \in C(\mathbb{Q}_2)$ such that $x_Q=3$, and
\[(3;3-T_2;3-T_3)=\delta_2(Q- \infty)  \in \mathrm{Im}(\delta_{2}).\]
Finally, since $2=\tau^2+\tau^3$ in $\mathbb{Q}_2[\tau]/(\tau^2+2 \tau+2)$ and $p \equiv 1+\tau^2+\tau^3 + \tau^6 \pmod{{\tau}^7}$, 
we can check that 
$f(\tau^2(1+\tau^2)) \equiv \tau^{10}(1+\tau)^2(1+\tau^2)^2 \pmod{\tau^{15}}$, hence 
$f(\tau^2(1+\tau^2))$ is square in $\mathbb{Q}_2[\tau]/(\tau^2+2 \tau+2)$ by using Hensel's lemma.  
Therefore, we see that there exist $R$, $\ol{R}$ such that $x_R+x_{\ol{R}}=-8$, $x_Rx_{\ol{R}}=20$, 
and $R+\ol{R}-2\infty$ defines an element in  $J(\mathbb{Q}_2)$.
Thus,
\[ (5;3;2-2T_3)=\delta_2(R+\ol{R}-2\infty) \in \mathrm{Im}(\delta_{2}).\]

By \cref{2 torsion11}, it is sufficient to prove that the four elements in the statement are linearly independent.
By taking their first components into account, the four elements in the statement are non-trivial in $\mathrm{Im}(\delta_{2})$,
and the only possible relation is
\begin{align*}
(2;-T_2;-T_3)  (6;T_2;2)  (3;3-T_2;3-T_3)=\1.  
\end{align*}
However, this is impossible because
\[T_2-3 \equiv 1+T_2 \not\equiv 1 \pmod{T_2^2}.\] 
Therefore, we obtain the assertion. 

By the first statement, to prove the second, it is sufficient to show that
\[\begin{cases}
\val_2(2;-T_2;-T_3)=((2);(2,T_2);(2))_2,\\
\val_2(6;T_2;2)=((2);(2,T_2);(2))_2,\\ 
\val_2(3;3-T_2;3-T_3)=\1_2,\\
\val_2(5;3;2-2T_3)=((1);(1);(2))_2. 
\end{cases}\]
First, the first components are obvious. 
Since $T_2$ is a uniformizer in $L^{(2)}_2$, we obtain the second components. 
Since $2$ is a uniformizer in $L^{(3)}_2$ we obtain the third components.
\end{proof}

\begin{corollary} \label{V11}
The following four elements of $L^{\times}/ L^{\times 2}$ form an $\mathbb{F}_{2}$-basis of $\Ker(\val)$: 
 \begin{align*}
 (1;1;-1),&& (1;-1;1),&& (-1;1;1),&& (1;2;1).
 \end{align*}
\end{corollary}

\begin{proof}
By the exactly same manner as in \cref{V01}, we obtain a short exact sequence 
\[1 \to \mathcal{O}_L^{\times}/\mathcal{O}_L^{\times 2} \to \Ker(\val) \to \Cl(L)[2] \to 0.\] 
First, $(1;1;-1)$, $(1;-1;1)$ and $(-1;1;1)$ form an $\mathbb{F}_2$-basis of $\mathcal{O}_L^{\times}/\mathcal{O}_L^{\times 2}$.
Since the ring of integers of $L^{(1)} \simeq \mathbb{Q}$ is a PID, we have $\Cl(L^{(1)})[2]=0$.
\cref{2Cl} implies that $\Cl(L^{(2)})[2]$ is generated by the class of $(2,T_2)$. 
By \cref{Genus theory} $(2)$, we have $\Cl(L^{(3)})[2]=0$. 
Therefore, by a diagram chasing, we can check that the image of $(1;2;1)$ in $\Cl(L)[2] $ is non-trivial.
This completes the proof.
\end{proof}

\begin{lemma} \label{W11}
The following three elements of $I(L)/I(L)^2$ form an $\mathbb{F}_{2}$-basis of $W$: 
\begin{align*} ((1);(1);(2))_2 \times \1_p, &&
 ((2);(2,T_2);(2))_2 \times ((1);(p,T_2);(T_3))_p, &&
\1_2 \times ((p);(1);(T_3))_p.
 \end{align*}
\end{lemma}

\begin{proof}
\cref{Jp11,J211}, the following four elements of $I(L)/I(L)^2$ form an $\mathbb{F}_2$-basis of $G$.
\begin{align*}
 ((2);(2,T_2);(2))_2 \times \1_p,&& ((1);(1);(2))_2 \times \1_p, && \\
 \1_2 \times ((1);(p,T_2); (T_3))_p,&& \1_2 \times ((p);(p,T_2);(1))_p.
 \end{align*}
Since the images of 
\begin{align*}
((1);(1);(2))_2 \times \1_p,&&  ((2);(2,T_2);(2))_2 \times ((1);(p,T_2);(T_3))_p,&& \1_2 \times ((p);(1);(T_3))_p. 
\end{align*}
in $\Cl(L)$ are trivial, the three elements in the statement lie in $W$.
Moreover, we see that they are linearly independent.
Finally, by \cref{2Cl}, we have $((2);(2,T_2);(2))_2 \times \1_p \not\in W$.
Therefore, we obtain the assertion.
\end{proof}

By the exactly same manner as in \cref{H01}, the following lemma follows from \cref{V11,W11}. 
\begin{lemma} \label{H11}
The following seven elements of $L^{\times}/L^{\times 2}$ form an $\mathbb{F}_{2}$-basis of ${\val}^{-1}(G)$: 
\begin{align*}
(1;1;-1),&& (1;-1;1),&& (-1;1;1), && (1;2;1),&& (1;1;2),&& (2;T_2;T_3),&& (p;1;T_3).
\end{align*}
\end{lemma}

\begin{lemma} \label{H+JS11}
Set elements of $L_2^{\times}/L_2^{\times 2} \times L_p^{\times}/L_p^{\times 2}$ as follows: 
\begin{align*}
d_1 &:= (2;-T_2;-T_3)_2 \times\1_p, \\
d_2 &:= (6;T_2;2)_2 \times\1_p, \\
d_3 &:= (3;3-T_2;3-T_3)_2 \times\1_p, \\
d_4 &:= (5;3; 2-2T_3)_2 \times\1_p, \\
d_5 &:= \1_2 \times(2;-T_2;-T_3)_p,\\
d_6 &:=\1_2 \times(2p;T_2;2)_p,\\
h_1 &:= \res_S(1;1;-1)=(1;1;-1)_2 \times(1;1;-1)_p, \\
h_2 &:= \res_S(1;-1;1)=(1;-1;1)_2 \times(1;-1;1)_p, \\
h_3 &:= \res_S(-1;1;1)=(-1;1;1)_2 \times(-1;1;1)_p, \\
h_4 &:= \res_S(1;2;1)=(1;2;1)_2 \times(1;2;1)_p, \\
h_5 &:= \res_S(1;1;2)=(1;1;2)_2 \times(1;1;2)_p, \\
h_6 &:= \res_S(2;T_2;T_3)=(2;T_2;T_3)_2 \times(2;T_2;T_3)_p, \\
h_7 &:= \res_S(p;1;T_3)=(3;1;T_3)_2 \times(p;1;T_3)_p.
\end{align*}
\begin{enumerate}
\item We have $d_1d_5h_1h_2h_6=d_2d_6h_7 = \1_2 \times \1_p$.
\item $d_1, \ldots, d_4$, $h_1, \ldots, h_7$ form an $\mathbb{F}_2$-basis of $(\mathrm{Im}(\delta_{2}) \times\mathrm{Im}(\delta_{p}))+\res_S({\val}^{-1}(G))$.
\end{enumerate}
\end{lemma}

\begin{proof}
\begin{enumerate}
\item We can check it by direct calculation.
\item By $(1)$, \cref{Jp11,J211,H11}, $(\mathrm{Im}(\delta_{2}) \times\mathrm{Im}(\delta_{p}))+\res_S({\val}^{-1}(G))$ is generated by $d_1, \ldots d_4, h_1,  \ldots , h_7$. 
Set
\begin{align*}
d_1' &:= d_1d_2d_3 =(1;-3+T_2;2T_3(T_3-3))_2 \times\1_p, \\ 
d_2 &= (6;T_2;2)_2 \times\1_p, \\
d_3 &=  (3;3-T_2;3-T_3)_2 \times\1_p, \\
d_4 &= (5;3; 2-2T_3)_2 \times\1_p, \\
h_1' &:= h_1h_5=(1;1;-2)_2 \times\1_p,\\
h_2' &:= h_2h_4 = (1;-2;1)_2 \times\1_p, \\
h_3' &:= h_3 h_6 =(-2;T_2;T_3)_2 \times(1;T_2;T_3)_p,\\
h_4&=\res_S(1;2;1)=(1;2;1)_2 \times(1;2;1)_p, \\
h_5&=\res_S(1;1;2)=(1;1;2)_2 \times(1;1;2)_p, \\
h_6 &= \res_S(2;T_2;T_3)=(2;T_2;T_3)_2 \times(2;T_2;T_3)_p, \\
h_7 &:= \res_S(p;1;T_3)=(3;1;T_3)_2 \times(p;1;T_3)_p.
\end{align*}
It is sufficient to prove that the above eleven elements are linearly independent.
\begin{enumerate}
\item By taking the first components at $v=p$ into account, we see that there is no relation containing $h_6$ and $h_7$. 
\item By taking the second components at $v=p$ into account, we see that there is no relation containing $h_3'$ and $h_4$.
\item By taking the third components at $v=p$ into account, we see that there is no relation containing $h_5$.
\item By taking the first components at $v=2$ into account, we see that there is no relation containing $d_2$, $d_3$ and $d_4$.
\item By taking the second components at $v=2$ into account, we see that there is no relation containing $d_1'$ and $h_2'$ because
\begin{itemize} 
\item $-3+T_2 \equiv T_2+1 \pmod{T_2^2}$, 
\item $-2$ is non-square in $L_2^{(2)} \simeq \mathbb{Q}_2(\sqrt{-6})$,
\item $-2(-3+T_2) \equiv T_2^2(T_2+1) \pmod{T_2^4}$.
\end{itemize}
\item Finally, by taking the third components at $v=2$ into account, we see that there is no relation containing $h_1'$. 
\end{enumerate}
\end{enumerate}
\end{proof}
Recall that 
\begin{align*}
\dim \Sel(\mathbb{Q},J^{(p; i, j)}) =& \dim {\val}^{-1}(G) + \dim(\mathrm{Im}(\delta_{2}) \times\mathrm{Im}(\delta_{p})) \\
&- \dim((\mathrm{Im}(\delta_{2}) \times\mathrm{Im}(\delta_{p}))+\res_S({\val}^{-1}(G)))
\end{align*}
and \[J(\mathbb{Q})/2J(\mathbb{Q}) \simeq \mathbb{F}_2^{\rank(J(\mathbb{Q}))} \oplus J(\mathbb{Q})[2].\]
Therefore, by \cref{2 torsion11,Jp11,J211,H11,H+JS11}, we obtain
\begin{align*}
\rank(J(\mathbb{Q}))  \leq \dim \Sel(\mathbb{Q},J)-\dim J(\mathbb{Q})[2]=7+6 -11-2=0. 
\end{align*}

This completes the proof of \cref{rank=0} $(2)$.


\subsection{Case $(i,j)=(0,2)$, $p \equiv 3 \pmod{8}$} 

Suppose that $p \equiv 3 \pmod{8}$. 
Then, we have the following irreducible decompositions:
\[\begin{cases}
L_2=\mathbb{Q}_2[T_1]/(T_1) \times \mathbb{Q}_2[T_2]/(T_2^2+p^2) \times \mathbb{Q}_2[T_3]/(T_3^2+2p^2), \\
L_p=\mathbb{Q}_p[T_1]/(T_1) \times \mathbb{Q}_p[T_2]/(T_2^2+p^2) \times \mathbb{Q}_p \times \mathbb{Q}_p.
\end{cases}\] 
Here, we fix an element $\alpha \in \mathbb{Q}_p$ such that ${\alpha}^2=-2$, which is denoted by $\sqrt{-2}$.
Moreover, we also fix an isomorphism $L_p^{(3)} \simeq 
\mathbb{Q}_p \times \mathbb{Q}_p$ which sends $T_3$ to $(p\sqrt{-2},-p\sqrt{-2})$. 
According to this irreducible decomposition, we denote each element in $L_p$ by the form $(\alpha_1;\alpha_2;\alpha_{3,1},\alpha_{3,2})$.

\begin{lemma}  \label{2 torsion02}
The following two elements form an $\mathbb{F}_{2}$-basis of $J(\mathbb{Q})[2]$ and $J(\mathbb{Q}_{v})[2]$ for $v = 2, \infty$: 
\begin{align*}
 (0,0) &- \infty, & \sum_{\substack{P \in C(\ol{\mathbb{Q}}) \\ x_{P}^{2}+p^2 = 0}}P &- 2\infty. 
 \end{align*}
On the other hand, the following three elements form an $\mathbb{F}_{2}$-basis of $J(\mathbb{Q}_{p})[2]$:
\begin{align*}
(0,0) &- \infty,& (p\sqrt{-2}, 0) &-\infty,&  (-p\sqrt{-2}, 0)&-\infty.
 \end{align*}
In particular, we have the following table.
\begin{table}[ht]
    \begin{tabular}{|l||l|l|} \hline
  $v$   & $\dim J(\mathbb{Q}_{v})[2]$ & $\dim \mathrm{Im}(\delta_{v}) $ 
  \\ \hline \hline
     $2$   & $2$ & $4$      \\ \hline
     $p$   & $3$   & $3$       \\ \hline
      $\infty$   & $2$ & $0$   \\ \hline  
    \end{tabular}
\end{table}
\end{lemma}

\begin{proof}

The first and second statements follow from \cite[Lemma 5.2]{Stoll2014}.
Note that
\begin{itemize}
\item neither $-1$ nor $-2$ is square in $\mathbb{Q}_{v}$ for $v = 2,  \infty$.
\item $-2$ is square in $\mathbb{Q}_{p}$ and $-1$ is not square in $\mathbb{Q}_{p}$.
\end{itemize}
The third statement follows from the following formula (cf. \cite[p. 451, proof of Lemma 3]{FPS}).
\[\dim_{\mathbb{F}_2} \mathrm{Im}(\delta_{v}) =\dim_{\mathbb{F}_2} J({\mathbb{Q}_{v}})[2]
\begin{cases}
+0 & (v \neq 2, \infty), \\
+2 & (v = 2), \\
-2 & (v = \infty). \\
\end{cases}\]
\end{proof}

\begin{lemma} 
\label{Jp02} 
The following three elements form a basis of $L_p^{\times}/L_p^{\times 2}$ form an $\mathbb{F}_{2}$-basis of $\mathrm{Im}(\delta_{p})$:
 \begin{align*}
  (2;T_2;-p \sqrt{-2},p \sqrt{-2}),&& (p\sqrt{-2};p \sqrt{-2}-T_2;1,-p\sqrt{-2}), && (-1;1;p \sqrt{-2}, -p\sqrt{-2}). 
 \end{align*}
In particular, the following three elements of $I_p(L)/I_p(L)^2$ form a basis of $G_p$:
\begin{align*}
((1); (p); (p),(p))_p,&& ((p); (p); (1),(p))_p,&& ((1); (1); (p),(p))_p.
\end{align*}
\end{lemma}
\begin{proof}
\cref{image of 2-torsion} implies that 
\begin{align*}
 \delta_p((0,0)-\infty) &= -T+(T^2+p^2)(T^2+2p^2)=(2;T_2;-p \sqrt{-2},p \sqrt{-2}),\\
 \delta_p((p\sqrt{-2},0)-\infty) &= -(T-p\sqrt{-2})+T(T+p\sqrt{-2})(T^2+p^2) \\
 &= (p\sqrt{-2};p \sqrt{-2}-T_2;1,-p\sqrt{-2}),\\
  \delta_p \left(\sum_{\substack{P \in C(\ol{\mathbb{Q}}) \\ x_{P}^{2}+2p^2 = 0}}P- 2\infty \right) &=(T^2+2p^2)-T(T^2+p^2) = (-1;1;p \sqrt{-2}, -p\sqrt{-2}).
\end{align*}

Hence, the above three elements lie in $\mathrm{Im}(\delta_{p})$. 
By taking their images in $G_p$ into account, we see that they are linearly independent.
By \cref{2 torsion02}, this completes the proof.
\end{proof}

\begin{lemma} 
\label{J202}
The following four elements of $L_2^{\times}/L_2^{\times 2}$ form an $\mathbb{F}_{2}$-basis of $\mathrm{Im}(\delta_{2})$:
\begin{align*}
(2;T_2;-T_3), && (1;T_2;-1),&&  (-2;2+T_2;-2-T_3), && (1;1-4T_2;1). 
\end{align*}
In particular, the following element of $I_2(L)/I_2(L)^2$ forms an $\mathbb{F}_{2}$-basis of $G_2$: 
\[((2);(1);(T_3))_2.\]
 \end{lemma}
\begin{proof}
First, we show that the above four elements actually lie in $\mathrm{Im}(\delta_{2})$.
\cref{image of 2-torsion} implies that
\begin{align*}
\delta_2((0,0)-\infty) &= -T+(T^2+p^2)(T^2+2p^2) = (2;T_2;-T_3),\\ 
\delta_2 \left(\sum_{\substack{P \in C(\ol{\mathbb{Q}}) \\ x_{P}^{2}+p^2 = 0}}P - 2\infty \right) &= (T^2+p^2)-T(T^2+2p^2) = (1; T_{2}; -1).\\ 
\end{align*}
Hence, $(2;T_2;-T_3)$, $(1;T_2;-1) \in \mathrm{Im}(\delta_{2})$. 
\begin{itemize}
\item  Since $f(-2) \equiv 4 \pmod {32}$, there exists $Q \in C(\mathbb{Q}_2)$ such that $x_Q=-2$.
In particular, $(-2;2+T_2;-2-T_3)=\delta_2(Q-\infty)$ lies in $\mathrm{Im}(\delta_{2})$.
\item  Since $2^{10}f(1/4) \equiv 1 \pmod{8}$, there exists $R \in C(\mathbb{Q}_2)$ such that $x_R=1/4$.
In particular, $(1;1-4T_2;1)=\delta_2(R-\infty)$ lies in $\mathrm{Im}(\delta_{2})$.
\end{itemize}
Since $v_2(2)=v_2(-2)=1$, the first and third elements are non-trivial in $L^{\times}/L^{\times 2}$. 
Since $-1$ is non-trivial in $L_2^{(3) \times}/L_2^{(3) \times 2}$, the second element is  non-trivial in $L^{\times}/L^{\times 2}$. 
In what follows, we show that $1-4T_2$ is non-trivial in $L^{(2) \times}/L^{(2) \times 2}$ by contradiction.

Let $t$ be a uniformizer in $L_2^{(2)}$, and $b_1$, $b_2 \in \{0,1\}$ such that $2 \equiv t^2+b_1t^3 \pmod{t^4}$.  
Suppose that $1-4T_2 \equiv (1+a_1t+a_2t^2+\cdots)^2 \pmod{t^5}$ with $a_1$, $a_2 \in \{0,1\}$.
Then, we have \[1+t^4 \equiv 1+a_1^2t^2+a_1t^3+(a_2+a_2^2+a_1b_1)t^4 \pmod{t^5},\]
which implies that $a_1=0$ and $a_2+a_2^2=1 \pmod{2}$, a contradiction.
Thus, $(1;1-4T_2;1)$ is non-trivial in $L^{\times}/L^{\times 2}$. 

Finally, by taking the first and third components into account, the four elements in the statement are linearly independent.
\end{proof}

\begin{corollary} \label{V02}
The following three elements of $L^{\times}/ L^{\times 2}$ form an $\mathbb{F}_{2}$-basis of $\Ker(\val)$: 
 \begin{align*}
 (1;1;-1),&& (1;T_2/p;1),&& (-1;1;1).
 \end{align*}
\end{corollary}

\begin{proof}
By the exactly same manner as in \cref{V01,V11}, we obtain a short exact sequence
\[1 \to \mathcal{O}_L^{\times}/\mathcal{O}_L^{\times 2} \to \Ker(\val) \to \Cl(L)[2] \to 0.\] 
The lemma follows from the facts that $(1;1;-1)$, $(1;T_2/p;1)$ and $(-1;1;1)$ form an $\mathbb{F}_2$-basis of $\mathcal{O}_L^{\times}/\mathcal{O}_L^{\times 2}$, and
the class number of $L^{(i)}$ is $1$ for every $i=0$, $1$, $2$.
\end{proof}

\begin{lemma} \label{W02} 
The following four elements of $I(L)/I(L)^2$ form an $\mathbb{F}_{2}$-basis of $W=G$: 
\begin{align*}
((2);(1);(T_3))_2 \times \1_p,&&
\1_2 \times ((1); (p); (p),(p))_p, &&\\
 \1_2 \times ((p); (p); (1),(p))_p, &&
\1_2 \times ((1); (1); (p),(p))_p.
\end{align*}
\end{lemma}

\begin{proof}
Since the class number of $L^{(i)}$ is $1$ for every $i=0$, $1$, $2$, we have $W=G$. Thus, the lemma follows from \cref{Jp02,J202}.
\end{proof}

Recall that there exists an exact sequence $1 \to \Ker(\val) \to {\val}^{-1}(G) \to W \to 1$. 
By \cref{V02,W02}, we have the following lemma.
\begin{lemma} \label{H02}
Fix $a,b \in \mathbb{Z}$ such that $p=a^2+2b^2$ in $L^{(3)}$. 
Then, the following seven elements of $L^{\times}/L^{\times 2}$ form an $\mathbb{F}_{2}$-basis of ${\val}^{-1}(G)$: 
\begin{align*}
(1;1;-1),&& (1;T_2/p;1),&& (-1;1;1), && (2;1;T_3),&& (1;p;p),&& (p;p;a+b \cdot T_3/p),&& (1;1;p).
\end{align*}  
\end{lemma}

\begin{lemma} \label{H+JS02}
Set elements of $L_2^{\times}/L_2^{\times 2} \times L_p^{\times}/L_p^{\times 2}$ as follows: 
\begin{align*}  
d_1 &:= (2;T_2;-T_3)_2 \times\1_p, \\
d_2 &:= (1;T_2;-1)_2 \times\1_p, \\
d_3 &:= (-2;2+T_2;-2-T_3)_2 \times\1_p, \\
d_4 &:= (1;1-4T_2;1)_2 \times\1_p, \\
d_5 &:= \1_2 \times(2; T_2; -p\sqrt{-2},p \sqrt{-2})_p, \\
d_6 &:= \1_2 \times(p \sqrt{-2}; p \sqrt{-2}-T_2; 1,-p \sqrt{-2})_p, \\
d_7 &:= \1_2 \times(-1; 1; p\sqrt{-2},-p \sqrt{-2})_p, \\
h_1 &:= \res_S(1;1;-1)=(1;1;-1)_2 \times(1;1;-1,-1)_p, \\
h_2 &:= \res_S(1; T_2/p;1)=(1;T_2/p;1)_2 \times(1;T_2/p;1,1)_p, \\
h_3 &:= \res_S(-1;1;1)=(-1;1;1)_2 \times(-1;1;1,1)_p, \\
h_4 &:= \res_S(2;1;T_3)=(2;1;T_3)_2 \times(2;1;p\sqrt{-2},-p\sqrt{-2})_p, \\
h_5 &:= \res_S(1;p;p)=(1;p;p)_2 \times(1;p;p,p)_p, \\
h_6 &:= \res_S(p;p;a+b\cdot T_3/p)=(p;p;a+b\cdot T_3/p)_2 \times(p;p;a+b\sqrt{-2},a-b\sqrt{-2})_p, \\
h_7 &:= \res_S(1;1;p)=(1;1;p)_2 \times(1;1;p,p)_p. 
\end{align*}  
\begin{enumerate}
\item We have $d_1 d_2 d_7 h_4 =d_2 d_5 d_7 h_1 h_2 h_5 h_7 = \1_2 \times\1_p.$
\item The twelve elements $d_3, \ldots, d_7$, $h_1, \ldots, h_7$ form an $\mathbb{F}_2$-basis of $(\mathrm{Im}(\delta_{2}) \times\mathrm{Im}(\delta_{p}))+\res_S({\val}^{-1}(G))$. 
\end{enumerate} 
\end{lemma}

\begin{proof} 
\begin{enumerate}
\item We can check it by direct calculation.
\item
By $(1)$ and \cref{Jp02,J202,H02}, $(\mathrm{Im}(\delta_{2}) \times\mathrm{Im}(\delta_{p}))+\res_S({\val}^{-1}(G))$ is generated by $d_3,  \ldots d_7, h_1, \ldots , h_7$.
Set
\begin{align*}
d_3 &=  (-2;2+T_2;-2-T_3)_2 \times\1_p,\\
d_4 &= (1;1-4T_2;1)_2 \times\1_p, \\
d_5' &:= d_5d_7h_1h_2h_5h_7 =(1;T_2;-1)_2 \times\1_p,\\
d_6 &= \1_2 \times(p \sqrt{-2}; p \sqrt{-2}-T_2; 1,-p \sqrt{-2})_p, \\
d_7' &:= d_7h_4 = (2;1;T_3)_2 \times\1_p,\\
h_1 &=(1;1;-1)_2 \times(1;1;-1,-1)_p, \\
h_2 &=(1;T_2/p;1)_2 \times(1;T_2/p;1,1)_p,\\
h_3' &:=h_3h_4= (-2;1;T_3)_2 \times(1;1;p\sqrt{-2},-p\sqrt{-2})_p, \\
h_4 &=(2;1;T_3)_2 \times(2;1;p\sqrt{-2},-p\sqrt{-2})_p, \\
h_5 &= (1;p;p)_2 \times(1;p;p,p)_p, \\
h_6 &= (p;p;a+b\cdot T_3/p)_2 \times(p;p;a+b\sqrt{-2},a-b\sqrt{-2})_p, \\
h_7 &=(1;1;p)_2 \times(1;1;p,p)_p.
\end{align*}
It is sufficient to prove that the above twelve elements are linearly independent.
\begin{enumerate}
\item By taking the first components at $v=p$ into account, we see that there is no relation containing $d_6$, $h_4$ and $h_6$.
\fn{Note that the relations among our generators do NOT depend on $p$ whenever $p \equiv 3 \pmod{8}$ 
although the representation of $h_6$ depends on $a$, $b \in \mathbb{Z}$ such that $p=a^2+2b^2$.
It is similar in the cases $p \equiv -3$, $-1 \pmod{8}$ (cf. \cref{H+JS02'}).
In contrast, when $p \equiv 1 \pmod{8}$, the relations DO depend on $a$, $b$, $c$, $d  \in \mathbb{Z}$ such that $p=a^2+2b^2=c^2+d^2$.
} 
\item By taking the second components at $v=p$ into account, we see that there is no relation containing $h_2$ and $h_5$.
\item By taking the third and fourth components at $v=p$ into account, we see that there is no relation containing $h_1$, $h_3'$ and $h_7$.
\item By taking the first components at $v=2$ into account, we see that there is no relation containing $d_3$ and $d_7'$.
\item By taking the third components at $v=2$ into account, we see that there is no relation containing $d_5'$.
\item Finally, by taking the second components at $v=2$ into account, we see that there is no relation containing $d_4$.
\end{enumerate}
\end{enumerate}
\end{proof}
Recall that 
\begin{align*}
\dim \Sel(\mathbb{Q},J^{(p; i, j)}) =& \dim {\val}^{-1}(G) + \dim(\mathrm{Im}(\delta_{2}) \times\mathrm{Im}(\delta_{p})) \\
&- \dim((\mathrm{Im}(\delta_{2}) \times\mathrm{Im}(\delta_{p}))+\res_S({\val}^{-1}(G)))
\end{align*}
and\ \[J(\mathbb{Q})/2J(\mathbb{Q}) \simeq \mathbb{F}_2^{\rank(J(\mathbb{Q}))} \oplus J(\mathbb{Q})[2].\]
Therefore, by \cref{2 torsion02,Jp02,J202,H02,H+JS02}, we obtain
\begin{align*}
\rank(J(\mathbb{Q}))  \leq \dim \Sel(\mathbb{Q},J)-\dim J(\mathbb{Q})[2]=7+7 -12-2=0. 
\end{align*}

This completes the proof of \cref{rank=0} $(3)$.


\subsection{Case $(i,j)=(0,2)$, $p \equiv -3 \pmod{8}$} 
Suppose that $p \equiv - 3 \pmod{8}$. 
Then, we have the following irreducible decompositions:
\[\begin{cases}
L_2=\mathbb{Q}_2[T_1]/(T_1) \times \mathbb{Q}_2[T_2]/(T_2^2+p^2) \times \mathbb{Q}_2[T_3]/(T_3^2+2p^2), \\
L_p=\mathbb{Q}_p[T_1]/(T_1) \times \mathbb{Q}_p \times \mathbb{Q}_p \times \mathbb{Q}_p[T_{3}]/(T_3^2+2p^2).
\end{cases}\] 
Here, we fix an element $\beta \in \mathbb{Q}_p$ such that ${\beta}^2=-1$, which is denoted by $\sqrt{-1}$.
Moreover, we also fix an isomorphism $L_p^{(2)} \simeq \mathbb{Q}_p \times \mathbb{Q}_p$  which sends $T_2$ to $(p\sqrt{-1},-p\sqrt{-1})$.  
According to this irreducible decomposition, we denote each element in $L_p$ by the form $(\beta_1;\beta_{2,1},\beta_{2,2};\beta_{3,2})$.

\begin{lemma} \label{2 torsion02'}
The following two elements form an $\mathbb{F}_{2}$-basis of $J(\mathbb{Q})[2]$ and $J(\mathbb{Q}_{v})[2]$ for $v = 2, \infty$: 
\begin{align*}
(0,0) &- \infty,& \sum_{\substack{P \in C(\ol{\mathbb{Q}}) \\ x_{P}^{2}+p^2 = 0}}P &- 2\infty.
\end{align*}
On the other hand, the following three elements form an $\mathbb{F}_{2}$-basis of $J(\mathbb{Q}_{p})[2]$:
\begin{align*}
(0,0) &- \infty, & (p\sqrt{-1}, 0)&- \infty, &  (-p\sqrt{-1}, 0)&-\infty.
\end{align*}

In particular, we have the following table.
\begin{table}[ht]
    \begin{tabular}{|l||l|l|} \hline
  $v$   & $\dim J(\mathbb{Q}_{v})[2]$ & $\dim \mathrm{Im}(\delta_{v}) $ 
  \\ \hline \hline
     $2$   & $2$ & $4$      \\ \hline
    $p$   & $3$   & $3$       \\ \hline
      $\infty$   & $2$ & $0$   \\ \hline  
    \end{tabular}
\end{table}
\end{lemma}
\begin{proof}
The first and second statements follow from \cite[Lemma 5.2]{Stoll2014}.
Note that
\begin{itemize}
\item neither $-1$ nor $-2$ is square in $\mathbb{Q}_{v}$ for $v = 2,  \infty$.
\item $-1$ is square in $\mathbb{Q}_{p}$ and $-2$ is not square in $\mathbb{Q}_{p}$.
\end{itemize}
The third statement follows from the following formula (cf. \cite[p. 451, proof of Lemma 3]{FPS}).
\[\dim_{\mathbb{F}_2} \mathrm{Im}(\delta_{v}) =\dim_{\mathbb{F}_2} J({\mathbb{Q}_{v}})[2]
\begin{cases}
+0 & (v \neq 2, \infty), \\
+2 & (v = 2), \\
-2 & (v = \infty). \\
\end{cases}\]
\end{proof}

\begin{lemma} 
\label{Jp02'} 
The following three elements of $L_p^{\times}/L_p^{\times 2}$ form an $\mathbb{F}_{2}$-basis of $\mathrm{Im}(\delta_{p})$.
\begin{align*}
(2;p\sqrt{-1},p\sqrt{-1};T_3),&& (p\sqrt{-1};1,2p\sqrt{-1};p\sqrt{-1}-T_3),&& (1;p\sqrt{-1},p\sqrt{-1};1). 
\end{align*}
In particular, the following three elements of $I_p(L)/I_p(L)^2$ form an $\mathbb{F}_2$  basis of $G_p$:
\begin{align*} 
((1); (p),(p);(p))_p,&& ((p);(1),(p); (p))_p,&& ((1); (p),(p); (1))_p.
\end{align*}
\end{lemma}
\begin{proof}
\cref{image of 2-torsion} implies that
\begin{align*}
\delta_p((0,0)-\infty) &= -T+(T^2+p^2)(T^2+2p^2) =(2;p\sqrt{-1},p\sqrt{-1};T_3), \\
\delta_p((p\sqrt{-1},0)-\infty) &= (T-p\sqrt{-1})+T(T+p\sqrt{-1})(T^2+2p^2) \\
&=  (p\sqrt{-1};2,2p\sqrt{-1};p\sqrt{-1}-T_3),\\
\delta_p \left(\sum_{\substack{P \in C(\ol{\mathbb{Q}}) \\x_{P}^{2}+p^2 = 0}}P - 2\infty \right) &= (T^2+p^2)-T(T^2+2p^2)=(1;p\sqrt{-1},p\sqrt{-1};1).
\end{align*}

Hence, the above three elements lie in $\mathrm{Im}(\delta_{p})$. 
By taking their images in $G_p$ into account, we see that they are linearly independent.
By \cref{2 torsion02'}, this completes the proof.
\end{proof}

\begin{lemma} 
\label{J202'} 
The following four elements of $L_2^{\times}/L_2^{\times 2}$ form an $\mathbb{F}_{2}$-basis of $\mathrm{Im}(\delta_{2})$:
\begin{align*}
(2;T_2;-T_3), && (1;T_2;-1), && (-2;2+T_2;-2-T_3), && (1;1-4T_2;1). 
\end{align*}
In particular, the following element of $I_2(L)/I_2(L)^2$ forms an $\mathbb{F}_2$-basis of $G_2$:
\[((2);(1);(T_3))_2.\]
 \end{lemma}
\begin{proof}
See the proof of \cref{J202}. 
\end{proof}

\begin{corollary} \label{V02'}
The following three elements of $L^{\times}/ L^{\times 2}$ form an $\mathbb{F}_{2}$-basis of $\Ker(\val)$: 
 \begin{align*}
 (1;1;-1),&& (1;T_2/p;1),&& (-1;1;1).
 \end{align*}
\end{corollary}
\begin{proof}
See the proof of \cref{V02}.
\end{proof}

\begin{lemma} \label{W02'} 
The following four elements of $I(L)/I(L)^2$ form an $\mathbb{F}_{2}$-basis of $W=G$: 
\begin{align*}
 ((2);(1);(T_3))_2 \times \1_p,&& \1_2 \times ((1); (p),(p);(p))_p, \\
 && \1_2 \times ((p);(1),(p); (p))_p,&& \1_2 \times ((1); (p),(p); (1))_p.
 \end{align*}
\end{lemma}
\begin{proof}
See the proof of \cref{W02}.
\end{proof}

Recall that there exists an exact sequence $1 \to \Ker(\val) \to {\val}^{-1}(G) \to W \to 1$. 
By \cref{V02',W02'}, we have the following lemma.
\begin{lemma} \label{H02'}
Fix $c,d \in \mathbb{Z}$ such that $p=c^2+d^2$ in $L^{(3)}$.
Then, the following seven elements of $L^{\times}/L^{\times 2}$ form an $\mathbb{F}_{2}$-basis of ${\val}^{-1}(G)$: 
\begin{align*}
(1;1;-1),&& (1;T_2/p;1),&& (-1;1;1),&& (2;1;T_3), && (1;p;p),&& (p;c+d \cdot T_2/p;p), && (1;p;1).
\end{align*}
\end{lemma}

\begin{lemma} \label{H+JS02'}
Set elements of $L_2^{\times}/L_2^{\times 2} \times L_p^{\times}/L_p^{\times 2}$ as follows: 
\begin{align*} 
d_1 &:= (2;T_2;-T_3)_2 \times\1_p, \\
d_2 &:= (1;T_2;-1)_2 \times\1_p, \\
d_3 &:= (-2;2+T_2;-2-T_3)_2 \times\1_p, \\
d_4 &:= (1;1-4T_2;1)_2 \times\1_p, \\
d_5 &:= \1_2 \times(2;p\sqrt{-1},p\sqrt{-1};T_3)_p, \\
d_6 &:= \1_2 \times(p\sqrt{-1};1,2p\sqrt{-1};p\sqrt{-1}-T_3)_p, \\
d_7 &:= \1_2 \times(1;p\sqrt{-1},p\sqrt{-1};1)_p, \\
h_1 &:= \res_S(1;1;-1)=(1;1;-1)_2 \times\1_p, \\
h_2 &:= \res_S(1;T_2/p;1)=(1;T_2/p;1)_2 \times(1;\sqrt{-1}, \sqrt{-1};1)_p, \\ 
h_3 &:= \res_S(-1;1;1)=(-1;1;1)_2 \times\1_p, \\
h_4 &:= \res_S(2;1;T_3)=(2;1;T_3)_2 \times(2;1,1;T_3)_p, \\
h_5 &:= \res_S(1;p;p)=(1;p;p)_2 \times(1;p,p;p)_p, \\
h_6 &:= \res_S(p;c+d\sqrt{-1};p)=(p;c+d \cdot T_2/p;T_3)_2 \times(p;c+d\sqrt{-1},c-d\sqrt{-1};T_3)_p, \\
h_7 &:= \res_S(1;p;1)=(1;p;1)_2 \times(1;p,p;1)_p.
\end{align*}
\begin{enumerate}
\item We have $d_1 d_2 d_7 h_4 =d_2 d_5 d_7 h_1 h_2 h_5 h_7 = \1_2 \times\1_p.$
\item The twelve elements $d_3, \ldots, d_7$, $h_1, \ldots, h_7$ form an $\mathbb{F}_2$-basis of $(\mathrm{Im}(\delta_{2}) \times\mathrm{Im}(\delta_{p}))+\res_S({\val}^{-1}(G))$. 
\end{enumerate}
\end{lemma}

\begin{proof} 
\begin{enumerate}
\item We can check it by direct calculation.
\item
By $(1)$ and \cref{Jp02',J202',H02'}, $(\mathrm{Im}(\delta_{2}) \times\mathrm{Im}(\delta_{p}))+\res_S({\val}^{-1}(G))$ is generated by $d_3,  \ldots d_7, h_1, \ldots , h_7$.
Set
\begin{align*}
d_3' &:= d_3 d_5 d_7h_3h_4=(1;2+T_2;(-2-T_3)T_3)_2 \times\1_p,\\ 
d_4 &= (1;1-4T_2;1)_2 \times\1_p, \\
d_5' &:= d_5d_7h_4=(2;1;T_3)_2 \times\1_p, \\
d_6 &= \1_2 \times(p\sqrt{-1};1,2p\sqrt{-1};p\sqrt{-1}-T_3)_p, \\
d_7' &:= d_7h_2h_7=(1;T_2;1)_2 \times\1_p, \\
h_1 &=(1;1;-1)_2 \times\1_p, \\
h_2 &= (1;T_2/p;1)_2 \times(1;\sqrt{-1}, \sqrt{-1};1)_p, \\ 
h_3 &= (-1;1;1)_2 \times\1_p, \\
h_4 &=(2;1;T_3)_2 \times(2;1,1;T_3)_p, \\
h_5 &= (1;p;p)_2 \times(1;p,p;p)_p, \\
h_6 &= (p;c+d \cdot T_2/p;T_3)_2 \times(p;c+d\sqrt{-1},c-d\sqrt{-1};T_3)_p, \\
h_7 &= (1;p;1)_2 \times(1;p,p;1)_p.
\end{align*}
It is sufficient to prove that the above twelve elements are linearly independent.
\begin{enumerate}
\item By taking the first components at $v=p$ into account, we see that there is no relation containing $d_6$, $h_4$ and $h_6$.
\item By taking the fourth components at $v=p$ into account, we see that there is no relation containing $h_5$.
\item By taking the second and third components at $v=p$ into account, we see that there is no relation containing $h_2$ and $h_7$.
\item By taking the first components at $v=2$ into account, we see that there is no relation containing $d_5'$, and $h_3$.
\item By taking the third components at $v=2$ into account, we see that there is no relation containing $d_3'$, and $h_1$.
\item Finally, by taking the second components at $v=2$ into account, we see that there is no relation containing $d_4$ and $d_7'$.
\end{enumerate}
\end{enumerate}
\end{proof}
Therefore, by \cref{2 torsion02',Jp02',J202',H02',H+JS02'}, we obtain
\begin{align*}
\rank(J(\mathbb{Q}))  \leq \dim \Sel(\mathbb{Q},J)-\dim J(\mathbb{Q})[2]=7+7 -12-2=0. 
\end{align*}

This completes the proof of \cref{rank=0} $(4)$.


\section{Application of the Lutz-Nagell type theorem for hyperelliptic curves}
Let $p$ be a prime number, $i, j \in {\mathbb{Z}}_{\geq 0}$, and $f(x)=x(x^2+2^ip^{j})(x^2+2^{i+1}p^{j})$.
Let $ C$ be a hyperelliptic curve defined by $y^2=f(x)$ and $J$ be its Jacobian variety.

By taking \cref{rank=0} into account, \cref{MT} is an immediate consequence of the following proposition.

\begin{proposition} \label{AJ}
Let $P \in C(\mathbb{Q}) \setminus \{\infty\}$ such that $\phi(P) \in J(\mathbb{Q})_{\mathrm{tors}}$. 
\begin{enumerate}
\item Suppose that $p \neq 3$. Then, $P=(0,0)$.
\item Suppose that $p =3$ and $(i,j) \not\equiv (2,2)$, $(3,2) \pmod{4}$.
Then, $P=(0,0)$.
\end{enumerate}
\end{proposition}

This proposition follows from the following Lutz-Nagell type theorem.
\begin{theorem} [{cf. \cite[Theorem 3]{Grant}}] \label{GLN}
Let $\phi:C \to J$ be the Abel-Jacobi map defined by $\phi(P)=[P- \infty]$.
If $P=(a,b) \in C(\mathbb{Q}) \setminus \{\infty\}$ and $\phi(P) \in J(\mathbb{Q})_{\mathrm{tors}}$, then $a,b \in \mathbb{Z}$ and either $b=0$ or $b^2 \mid \disc(f)$.
\end{theorem} 
Note that Grant \cite{Grant} proved the above Lutz-Nagell type theorem in more general settings.
\begin{proof}  [Proof of \cref{AJ}] 
\begin{enumerate}
\item
We prove this statement by contradiction.
Suppose that $P=(a,b) \in C(\mathbb{Q}) \setminus \{\infty\}$ satisfies $\phi(P) \in J(\mathbb{Q})_{\mathrm{tors}}$ and $b \neq 0$. 
Then, \cref{GLN} implies that $a$, $b \in \mathbb{Z}$ and $b^2 \mid \disc(f)$, hence $b=\pm 2^kp^l$ for some $k$, $l \in \mathbb{Z}_{\geq 0}$. 
In particular, every prime divisor of $f(a)  \in \mathbb{Z}_{\geq 0}$ is $2$ or $p$.

First, if $l=0$, then $f(a)$ is a power of $2$, which contradicts that $1<(a^2+2^{i+1}p^j)/(a^2+2^ip^j)<2$. Thus, we may assume that $p$ is odd and $l>0$.

\begin{itemize}
\item Suppose that $j=0$. Then,  $\disc(f)$ is divisible only by $2$, i.e. $l=0$, which is impossible.
\end{itemize}
Thus, $j$ must be positive, and there exists $a_1 \in \mathbb{Z}$ such that $a=pa_1$, i.e.
\[ b^2=p^3a_1(pa_1^2+2^ip^{j-1})(pa_1^2+2^{i+1}p^{j-1}).\]
\begin{itemize}
\item Suppose that $j=1$.
Since $p$ is odd, neither $pa_1^2+2^i$ nor $pa_1^2+2^{i+1}$ is divisible by $p$. Thus, both $pa_1^2+2^i$ and $pa_1^2+2^{i+1}$ must be powers of $2$, which contradicts that $1<(pa_1^2+2^{i+1})/(pa_1^2+2^i)<2$.

\item Suppose that $j=2$.
Then,  \[b^2=p^5a_1(a_1^2+2^i)(a_1^2+2^{i+1}).\]
\begin{itemize}
\item Suppose that $i=0$. 
If $a_1$ is odd, then  the right hand side is divisible by $2$ exactly odd times, which is impossible. 
If $a_1$ is even, then $a_1^2+1>1$ and $(a_1^2+2)/2>1$ must be coprime powers of $p$, which is impossible. 
\item Suppose that $i=1$. 
If $a_1$ is odd, then $a_1^2+2>1$ and $a_1^2+4>1$ must be coprime powers of $p$, which is impossible.
If $a_1$ is even, then there exists $a_2 \in \mathbb{Z}$ such that $a_1=2a_2$, i.e \[b^2=2^4p^5a_2(2a_2^2+1)(a_2^2+1).\]
If $a_2$ is odd, then the right hand side is divisible by $2$ exactly odd times, which is impossible.
If $a_2$ is even, then $2a_2^2+1>1$ and $a_2^2+1>1$ must be coprime powers of $p$, which is impossible.
\item Suppose that $i=2$. 
If $a_1$ is odd, then $a_1^2+4>1$ and $a_1^2+8>1$ must be coprime powers of $p$, which is impossible.
If $a_1$ is even, then there exists $a_2 \in \mathbb{Z}$ such that $a_1=2a_2$, i.e. \[b^2=2^5p^5a_2(a_2^2+1)(a_2^2+2).\]
If $a_2$ is odd, then $a_2$ and $a_2^2+2$ are coprime powers of $p$. Thus $a_2=1$, and the right hand side is divisible by $p \neq 3$ exactly odd times, which is impossible.
If $a_2$ is even, then $a_2^2+1>1$ and $(a_2^2+2)/2>1$ must be coprime powers of $p$, which is impossible.

\item Suppose that $i=3$. 
If $a_1$ is odd, then $a_1^2+8>1$ and $a_1^2+16>1$ must be coprime powers of $p$, which is impossible.
If $a_1$ is even, then there exists $a_2 \in \mathbb{Z}$ such that $a_1=2a_2$, i.e. \[b^2=2^5p^5a_2(a_2^2+2)(a_2^2+4).\]
If $a_2$ is odd, then $a_2^2+2>1$ must be $a_2^2+4>1$ are coprime powers of $p$, which is impossible.
If $a_2$ is even, then there exists $a_3 \in \mathbb{Z}$ such that $a_2=2a_3$, i.e. \[b^2=2^9p^5a_3(2a_3^2+1)(a_3^2+1).\]
If $a_3$ is odd, then $a_3$ and $2a_3^2+1$ are coprime powers of $p$. Thus, $a_3=1$, and the right hand side is divisible by $p \neq 3$ exactly odd times, which is impossible.
If $a_3$ is even, then $2a_3^2+1>1$ and $a_3^2+1>1$ must be coprime powers of $p$, which is impossible.
\item Suppose that $i \geq 4$. 
If $a_1$ is odd, then $a_1^2+2^i>1$ and $a_1^2+2^{i+1}>1$ must be coprime powers of $p$, which is impossible.
If $a_1$ is even, then there exists $a_2 \in \mathbb{Z}$ such that $a_1=2a_2$, i.e. \[b^2=2^5p^5a_2(a_2^2+2^{i-2})(a_2^2+2^{i-1}).\]
If $a_2$ is odd, then $a_2^2+2^{i-2}>1$ and $a_2^2+2^{i-1}>1$ must be coprime powers of $p$, which is a impossible.
If $a_2$ is even, then there exist $a_3$, $b_1 \in \mathbb{Z}$ such that $(a_2,b)=(2a_3,2^5b_1)$, i.e. \[b_1^2=p^5a_3(a_3^2+2^{i-4})(a_3^2+2^{i-3}). \]
Therefore, by a simple induction on $i$, we obtain the claim.
\end{itemize}

\item The case $j=3$ is similar to the case $j=1$.
\item 
If $j \geq 4$, then there exist $a_2$, $b_1 \in \mathbb{Z}$ such that $(a_1,b)=(pa_2,p^5b_1)$, i.e.
\[b_1^2=a_2(a_2^2+2^ip^{j-4})(a_2^2+2^{i+1}p^{j-4}).\]
Therefore, by a simple induction on $j$, we obtain the claim.
\end{itemize}
This completes the proof.
\item The proof of $(2)$ is exactly similar to $(1)$. Note that, in the above proof, we used the condition $p \neq 3$ only in the cases $(i,j) \equiv (2,2)$, $(3,2) \pmod{4}$.
\end{enumerate}
\end{proof}

\section{Concluding remarks}

If a hyperelliptic curve $C$ covers an elliptic curve $E$ over $\mathbb{Q}$, then we can determine $C(\mathbb{Q})$ by determining $E(\mathbb{Q})$.
\fn{For example, $y^2=(x^2+1)(x^2+1+p)(x^2+1-p)$ with a prime $p \equiv 3 \pmod{8}$ has no rational solution.} 
However, our hyperelliptic curves in \cref{MT} do not cover elliptic curves over $\mathbb{Q}$. 
Indeed, the congruent zeta functions of $C^{(p;i,j)}$ with $p \neq 11$ over $\mathbb{F}_{11}$ are \[Z(C^{(p;i,j)}/\mathbb{F}_{11},T)=\frac{1+14T^2+121T^4}{(1-T)(1-11T)}, \]
whose numerators are irreducible in $\mathbb{Q}[T]$.
\fn{Similarly, by calculating the congruent zeta functions over $\mathbb{F}_{13}$, we can check that $C^{(11;i,j)}$ do not cover elliptic curves over $\mathbb{Q}$.}

On the other hand, it might be possible to prove \cref{MT} by determining the set of rational points of some elliptic curves over number fields covered by our hyperelliptic curves. 
There are two obstructions in this direction.
First, it is a non-trivial problem to find such elliptic curves of Mordell-Weil rank $0$. 
Secondly, since the unit group of number fields are infinite in general,
the Lutz-Nagell theorem is not sufficient to determine the torsion points of elliptic curves over number fields.

\appendix

\section{}
In this appendix, we prove the following proposition.
\begin{proposition} \label{ff}
Let $i, j\in {\mathbb{Z}}_{\geq 0}$ and $C^{(i, j)}$ be a hyperelliptic curve over $\mathbb{Q}(t)$ defined by
\[ C^{(i, j)} : y^2=x(x^2+2^it^{j})(x^2+2^{i+1}t^{j}). \]
Then, $C^{(i, j)}(\mathbb{Q}(t))=\{(0,0),\infty\}$.
\end{proposition}

This is a consequence of the following ``abc conjecture for polynomials".
\begin{theorem} [{cf. \cite[Theorem]{Snyder}}] \label{abc} 
Let $k$ be a field of characteristic $0$, and $a$, $b$ and $c \in k[x]$ be coprime polynomials such that $a+b=c$.
Suppose that $a$, $b$ and $c$ are not all constants. Then, $\deg (c) < n_0(abc)$,
where $n_0(abc)$ is the number of distinct zeroes of $abc$.
\end{theorem}

\begin{lemma} \label{Fermat} 
There exist no $X$, $Y$, $Z \in \mathbb{Z}_{>0}$ such that
\begin{enumerate}
\item $X^4+Y^4=Z^2$ or
\item $X^4+4Y^4=Z^2$.
\end{enumerate}
\end{lemma}

\begin{proof} [Proof of \cref{Fermat}]
$(1)$ is a well-know theorem by Fermat.

We prove $(2)$ by contradiction.

Suppose that there exist $X$, $Y$, $Z \in \mathbb{Z}_{>0}$ such that $X^4+4Y^4=Z^2$. 
If $\gcd(X,Y)=d \in \mathbb{Z}_{>0}$, then $Z$ is divisible by $d^2$.
Thus, by replacing $X$, $Y$, $Z$ by $X/d$, $Y/d$, $Z/d^2$ if necessary, 
we may assume that $\gcd(X,Y,Z)=1$.
Assume that $X$ is even. 
Since $\gcd(X^2/2,Y^2,Z/2)=1$, there exist $S$, $T \in \mathbb{Z}_{>0}$ such that  $\gcd(S,T)=1$ and
\[\begin{cases}
\frac{X^2}{2}=2ST,\\
Y^2=S^2-T^2,\\
\frac{Z}{2}=S^2+T^2.\\
\end{cases}\]
Since $\gcd(S,T)=1$, the first equality implies that there exist $U$, $V \in \mathbb{Z}_{>0}$ such that $S=U^2$ and $T=V^2$.
The second equality implies that $Y^2=U^4-V^4$.
Since $Y$ is odd, $U$ is odd and $V$ is even.
Therefore, there exists $S'$, $T' \in \mathbb{Z}_{>0}$ such that $\gcd(S',T')=1$ and
\[\begin{cases}
Y=S'^2-T'^2,\\
V^2=2S'T',\\
U^2=S'^2+T'^2.\\
\end{cases}\]
The second equality implies that there exist $U'$, $V' \in \mathbb{Z}_{>0}$ such that $S'=U'^2$ and $T'=V'^2$.
Therefore, $U^2=U'^4+V'^4$, which contradicts $(1)$.

The case when $X$ is odd is similar.
\end{proof}

\begin{proof} [Proof of \cref{ff}]
Recall that $C^{(i, j)} \subset \mathbb{P}^2_{(1,3,1)}$ is defined by the following weighted homogeneous equation: 
\[Y^2=XZ(X^2+2^it^{j}Z^2)(X^2+2^{i+1}t^{j}Z^2).\]
We may assume that $X$, $Y$, $Z \in \mathbb{Q}[t]$.
Put $t=s^4$, $U=s^{j}Y$ and $V=s^{2j}Z$. Then, we obtain the curve $\tilde{C}^{(i, j)}$ defined by the following equation over $\mathbb{Q}(s)$:
\[U^2=XV(X^2+2^iV^2)(X^2+2^{i+1}V^2).\]
It is sufficient to prove that $\tilde{C}^{(i, j)}(\mathbb{Q}(s)) = \{(0,0), \infty\}$.
Let $P:=[X:U:V] \in C^{(i, j)}(\mathbb{Q}(s))$.
We may assume that $X=X(s)$, $U = U(s)$, $V=V(s)$ are taken from $\mathbb{Z}[s]$ such that $\gcd(X, U, V) = 1$. 
Then, there exist coprime polynomials $U_1$, $U_2$, $U_3$, $U_4 \in \mathbb{Z}[s]$ such that
\[\begin{cases}
X=U_1^2,\\
V=U_2^2,\\
X^2+2^iV^2=U_3^2,\\
X^2+2^{i+1}V^2=U_4^2.
\end{cases}\] 
Since $i$ or $i+1$ is even, there exists $U_5 \in \mathbb{Z}[s]$ such that one of the following conditions holds.
\[\begin{cases}
U_1^4+U_5^4=U_3^2 & (i \equiv 0 \pmod{4}). \\
U_1^4+4U_5^4=U_4^2 & (i \equiv 1 \pmod{4}).\\
U_1^4+4U_5^4=U_3^2 &    (i \equiv 2 \pmod{4}).\\ 
U_1^4+U_5^4=U_4^2 &  (i \equiv 3 \pmod{4}).
\end{cases}\]
Suppose that $U_1^4+U_5^4=U_3^2$. 
Let $g := \gcd(U_1, U_3, U_5) \in \mathbb{Z}_{>0}$, $U_1' := U_1/g$, $U_3' := U_2/g^2$, $U_5' := U_5/g$.
If $U_1'$, $U_3'$, $U_5'$ are not all constants, then \cref{abc} implies that 
\[2\deg(U_3')< n_0(U_1'U_3'U_5') \leq \deg U_1' + \deg U_3' + \deg U_5'.\]
Thus, we have \[\deg(U_3'^2)<\deg (U_1'^2) + \deg (U_5'^2) \leq \deg (U_1'^4+U_5'^4), \] which is impossible.
Therefore, $U_1'$, $U_3'$, $U_5'$ are all integers such that $\gcd(U_1', U_3', U_5') = 1$.  
Thus, by \cref{Fermat} $(1)$, $(U_1',U_3',U_5')=(0,\pm 1,\pm 1)$ or $(\pm 1, \pm 1,0)$, i.e. $P=[0:0:1]$ or $[1:0:0]$ as desired.
The other cases are exactly similar. 
This completes the proof.
\end{proof}

\noindent {\bf Acknowledgement.}
The authors thank their advisor Ken-ichi Bannai for reading the draft and giving helpful comments.
The authors also thank him for warm and constant encouragement and professor Shuji Yamamoto for helpful comments.

\end{document}